\documentclass[12pt, oneside]{amsart}   	
\usepackage{geometry}
\usepackage{amsmath, amssymb}
\usepackage{mathrsfs}
\usepackage{epsfig}
\usepackage{epic,eepic}              		
\usepackage{graphicx}				
\newtheorem{thm}{Theorem}[section]
\newtheorem{prop}[thm]{Proposition}

\newtheorem{lem}[thm]{Lemma}
\newtheorem{cor}[thm]{Corollary}
\newtheorem{conj}[thm]{Conjecture}

\newtheorem{ques}[thm]{Question}

\theoremstyle{definition}
\newtheorem{defn}[thm]{Definition}

\numberwithin{equation}{section}

\title{the endpoint perturbed Brascamp-Lieb inequality with examples}
\author{Ruixiang Zhang}
\usepackage{fancyhdr}
\begin{document}

\maketitle
\begin{abstract}
We prove the folklore endpoint multilinear $k_j$-plane conjecture originated from the paper \cite{bennett2006multilinear} of Bennett, Carbery and Tao. Along the way we prove a more general result, namely the endpoint multilinear $k_j$-variety theorem. Finally, we generalize our results to the endpoint perturbed Beascamp-Lieb inequality using techniques in this paper.
\end{abstract}

\section{introduction}

\subsection{The Endpoint Multilinear $k_j$-Plane Theorem}

The multilinear $k_j$-plane conjecture was implicitly proved by Bennett, Carbery and Tao in \cite{bennett2006multilinear}, except for the endpoint case. In the first part of this paper we formulate it and prove the endpoint case. In fact we will prove the endpoint multilinear $k_j$-variety theorem, which is more general.

The proof uses the polynomial method. We will set up the polynomial like Guth did in his proof of the endpoint multilinear Kakeya Conjecture \cite{guth2010endpoint}. Then we make some crucial new observations and development of the theory, enabling us to estimate ``the quantitative interaction of the polynomial with itself'' in terms of its visibility. As a result, we are able to deal with the codimension difficulty and complete the proof. As with Guth's paper\cite{guth2010endpoint}, I feel here that the new ingredients of our proof are as interesting as the result itself.

The multilinear $k_j$-plane estimate is a natural generalization of the infamous multilinear Kakeya estimate. Albeit being weaker than the linear Kakeya, the multilinear Kakeya theorem and the methods it inspired recently had remarkable applications to classical harmonic analysis problems as well\cite{bourgain2011bounds}\cite{bourgain2013schrodinger}\cite{bourgain2013moment}\cite{guth2014restriction}\cite{bourgain2015proof}. See the beginning of \cite{guth2015short} for a good introduction.

The non-endpoint case of the multilinear Kakeya conjecture was proved by Bennett-Carbery-Tao in \cite{bennett2006multilinear} and later Guth \cite{guth2010endpoint} proved the endpoint case. We state the endpoint theorem of Guth as the following theorem.

\begin{thm}\label{Guthendpointthm}
For $1 \leq j \leq d$, let $\{T_{j, a}: 1\leq a \leq A(j)\}$ be a family of unit cylinders in $\mathbb{R}^d$, we denote $v_{j, a}$ to be the direction of the core line of cylinder $T_{j, a}$. Assume the core lines of cylinders from different families are ``quantitatively transversal'', i.e. for any $1 \leq a_j \leq A(j)$, we have $v_{1, a_1} \wedge v_{2, a_2} \wedge \cdots \wedge v_{d, a_d} \geq \theta >0$ where $\theta$ is fixed. Then we have
\begin{equation}
\int_{\mathbb{R}^d} (\prod_{j=1}^d \sum_{a=1}^{A(j)} \chi_{T_{j, a}})^{\frac{1}{d-1}} \lesssim_d \theta^{-\frac{1}{d-1}} \prod_{j=1}^d A(j)^{\frac{1}{d-1}}.
\end{equation}
\end{thm}

Guth's approach to prove Theorem \ref{Guthendpointthm} is very different from the approach of Bennett-Carbery-Tao. He was able to take a polynomial that approximate the intersection of tubes sufficiently well, along the way employing some nice tools and lemmas from algebraic topology and integral geometry, which was quite interesting.

In the Kakeya setting we have cylinders which are neighborhoods of lines. A natural analogue is to replace lines with higher dimensional affine subspaces and this will exactly be our multilinear $k_j$-plane setting. In Remark 5.4 of \cite{bennett2006multilinear}, the authors note that their techniques can be also used to obtain non-endpoint case of multilinear $k$-plane transform estimates that Oberlin and Stein consider in \cite{oberlin1982mapping}. There is also a $k$-plane version of Kakeya problem\cite{bourga1991besicovitch} that could be relevant here.

They did not state the result precisely and we will state what we can get from their proof below. If we go check the proof, similar techniques in \cite{guth2015short} can also give us the result. Here we allow subspaces of different dimensions and hence call the theorem a ``multilinear $k_j$-plane theorem''.

Before stating the theorem we introduce our terminology to describe ``higher dimensional'' analogue of cylinders.

\begin{defn}
In a space of dimension $d$, for any $1 \leq b <d$ define a \emph{$b$-slab} to be the Cartesian product of a $b$-dimensional ball $B_1$ and a $(d-b)$-dimensional ball $B_2$ (the spaces spanned by both balls are required to be orthogonal). The radius of $B_1$ will be called the \emph{size} of our $b$-slab and the radius of $B_2$ will be called the \emph{radius} of it. The Cartesian product of $B_1$ and the center of $B_2$ is called the \emph{core} of this $b$-slab.
\end{defn}

By the above definition, a $1$-slab is a cylinder. Its length is the size in our language. Our definitions of radius and core are consistent with familiar definitions for cylinders. As explained above, we call our theorem a $k_j$-plane theorem because when the size is large a $k$-slab looks flat and is like an object around a $k$-plane.

\begin{thm}[Multilinear $k_j$-Plane Theorem with $R^{\varepsilon}$ Loss\cite{bennett2006multilinear}]\label{multikjplanethmwithepsilonpowerloss}
Assume $R$ is a large positive number. Assume $K_1, K_2, \ldots, K_n \subsetneqq \{1, 2, \ldots, d\}$ are disjoint and $K_1 \bigcup \cdots \bigcup K_n = \{1, 2, \ldots, d\}$. Let $k_j = |K_j|$.

For $1 \leq j \leq n$, let $\{T_{j, a}: 1\leq a \leq A(j)\}$ be a family of $k_j$-slabs of size $\leq R$ and radius $1$. Assuming that for any $1 \leq a_j \leq A(j)$, the core of $T_{j, a_j}$ is on a $k_j$-plane that forms an angle $<\delta$ against the $k_j$-plane spanned by all $\mathbf{e}_i, i \in K_j$.

Then when $\delta>0$ is sufficiently small depending on $d$, we have
\begin{equation}
\int_{\mathbb{R}^d} (\prod_{j=1}^n \sum_{a=1}^{A(j)} \chi_{T_{j, a}})^{\frac{1}{n-1}} \lesssim_{\varepsilon, d} R^{\varepsilon}\prod_{j=1}^n A(j)^{\frac{1}{n-1}}.
\end{equation}
\end{thm}

When $n=d$ and $K_j = \{j\}$, this theorem is the multilinear Kakeya theorem with $R^{\varepsilon}$ loss, being the main theorem of \cite{bennett2006multilinear}. In \cite{guth2015short}, a simpler proof of this special case is also proven. The proof in \cite{guth2015short} can also be generalized easily to prove the whole Theorem \ref{multikjplanethmwithepsilonpowerloss}.

We can obtain various $k_j$-plane theorems by taking different $n$ and $K_j$ in Theorem \ref{multikjplanethmwithepsilonpowerloss}. As we saw in Theorem \ref{Guthendpointthm}, Guth \cite{guth2010endpoint} was able to remove the $R^{\varepsilon}$ in the multilinear Kakeya case. So in general we would also expect the removal of $R^{\varepsilon}$. Conceptually, this will allow us to have slabs with ``size $\infty$'' (that are actually $1$-neighborhoods of $k_j$ planes) in the theorem. It turns out to be true and will be proved in this paper.

\begin{thm}[Multilinear $k_j$-Plane Theorem]\label{multilinearkjplanethm}
Assumptions are the same as Theorem \ref{multikjplanethmwithepsilonpowerloss}, but no restriction on the size of slabs. We have
\begin{equation}
\int_{\mathbb{R}^d} (\prod_{j=1}^n \sum_{a=1}^{A(j)} \chi_{T_{j, a}})^{\frac{1}{n-1}} \lesssim_{d} \prod_{j=1}^n A(j)^{\frac{1}{n-1}}.
\end{equation}
\end{thm}

Theorem \ref{multilinearkjplanethm} has an affine-invariant version, just like the multilinear Kakeya case which was pointed out in \cite{carbery2013endpoint}. We will actually prove this version (Theorem \ref{affinemultilinearkjplanethm} below). Theorem \ref{multilinearkjplanethm} is a direct corollary of it.

In order to state the theorem, we introduce some notations. For any $q \leq d$ vectors $\mathbf{v}_1, \mathbf{v}_2, \ldots, \mathbf{v}_q$, we denote $|\mathbf{v}_1\wedge \mathbf{v}_2 \wedge \cdots \wedge \mathbf{v}_q|$ to be the volume of the parallelepiped generated by $\mathbf{v}_1, \mathbf{v}_2, \ldots, \mathbf{v}_q$. Moreover for any $m$ (affine) subspaces $V_1, V_2, \ldots, V_m$ with a total dimension $d$, we can define $|V_1 \wedge V_2 \wedge \cdots \wedge V_m|$ to be $|\mathbf{v}_{1, 1} \wedge \cdots \wedge \mathbf{v}_{1, d_1} \wedge \mathbf{v}_{2, 1} \wedge \cdots \wedge \mathbf{v}_{2, d_2} \wedge \cdots \wedge \mathbf{v}_{m, 1} \cdots \wedge \mathbf{v}_{m, d_m}|$, where $\{\mathbf{v}_{j, i} (1 \leq i \leq d_j)\}$ form an orthonormal basis of the linear subspace parallel to $V_j$.

\begin{thm}[Affine invariant Multilinear $k_j$-Plane Theorem]\label{affinemultilinearkjplanethm}

For $1 \leq j \leq n$, let $\{T_{j, a}: 1\leq a \leq A(j)\}$ be a family of $k_j$-slabs of radius $1$. Assume the core $k_j$-plane of $T_{j, a}$ is parallel to the linear subspace $H_{j, a}$. Then for any real numbers $\rho_{j, a}$ we have

\begin{equation}
\int_{\mathbb{R}^d} (\sum_{a_1 = 1}^{A(1)} \cdots \sum_{a_n = 1}^{A(n)} \prod_{j=1}^n \rho_{j, a_j} \chi_{T_{j , a_j}}(x) \cdot H_{1, a_1} \wedge \cdots \wedge H_{n, a_n})^{\frac{1}{n-1}} \mathrm{d}x \lesssim_{d} \prod_{j=1}^n (\sum_{a_j = 1}^{A(j)} |\rho_{j, a_j}|)^{\frac{1}{n-1}}.
\end{equation}
\end{thm}

Our Theorem \ref{affinemultilinearkjplanethm} has some application in the multilinear restriction theorem too. Assuming for each $1 \leq j \leq n$, $\Sigma_j : U_j \rightarrow \mathbb{R}^d$ is a smooth parametrization of a subset of smooth submanifold $\Omega_j$ whose closure is compact. Also assume $\sum_{j=1}^n \dim \Omega_j = d$. Here we assume $U_j$ is a neighborhood of the origin $0$. We can associate the \emph{extension operator} to $\Sigma_j$ as follows:

\begin{equation}
E_j g_j (\xi) = \int_{U_j} e^{2 \pi i \xi \cdot \Sigma_j (x)} g_j (x) \mathrm{d} x.
\end{equation}

Assuming $T_{\Sigma_1 (0)} \Omega_1 \wedge \cdots \wedge T_{\Sigma_n (0)} \Omega_n \neq 0$. Then just like the classical multilinear restriction case discussed in \cite{bennett2006multilinear}, we can form the endpoint multilinear restriction conjecture:

\begin{conj}[Endpoint Multilinear $k_j$-Restriction Conjecture]\label{endpointmultilinearkjrestriction}
Assume we have $\Sigma_j$ as above such that $T_{\Sigma_1 (0)} \Omega_1 \wedge \cdots \wedge T_{\Sigma_n (0)} \Omega_n \neq 0$. Then when the $U_j$'s are sufficiently small we have
\begin{equation}\label{conjmultikjrestriction}
\int_{\mathbb{R}^d} \prod_{j=1}^n |E_j g_j|^{\frac{2}{n-1}} \lesssim_{d} \prod_{j=1}^n \|g_j\|_{L^2 (U_j)}^{\frac{2}{n-1}}.
\end{equation}
\end{conj}

The methods in \cite{bennett2006multilinear} yield the following local variant of the conjectured (\ref{conjmultikjrestriction}) with $R^\varepsilon$-loss :

\begin{equation}\label{multikjrestrictionepsilonloss}
\int_{B(0, R)} \prod_{j=1}^n |E_j g_j|^{\frac{2}{n-1}} \lesssim_{d, \varepsilon} R^{\varepsilon} \prod_{j=1}^n \|g_j\|_{L^2 (U_j)}^{\frac{2}{n-1}}.
\end{equation}

We can use Theorem \ref{multilinearkjplanethm} to slightly improve (\ref{multikjrestrictionepsilonloss}). Using exactly the same proof techniques as in the proof of Theorem 4.2 in \cite{bennett2013aspects}, from Theorem \ref{multilinearkjplanethm} we deduce that there exists a $\kappa = \kappa(d) > 0$ such that:

\begin{equation}\label{multikjrestrictionlogloss}
\int_{B(0, R)} \prod_{j=1}^n |E_j g_j|^{\frac{2}{n-1}} \lesssim_{d} (\log R)^{\kappa} \prod_{j=1}^n \|g_j\|_{L^2 (U_j)}^{\frac{2}{n-1}}.
\end{equation}

\subsection{The Endpoint Perturbed Brascamp-Lieb Inequality}

Everything in the previous section including the Loomis-Whitney inequality and the multilinear $k_j$-plane Theorem, in their unperturbed version, is a special case of the Brascamp-Lieb inequality. In this paper we also generalize the Brascamp-Lieb inequality in the same way we do with the multilinear $k_j$-plane Theorem, with some new combinatorial ideas. We state our endpoint perturbed Brascamp-Lieb inequality in this section.

We first briefly review the Brascamp-Lieb inequality. We will mostly follow the notational convention in \cite{bennett2008brascamp} and \cite{bennett2010finite}, which are two important references in literature. Assume that in $\mathbb{R}^d$ we have $n$ linear surjections $B_j : \mathbb{R}^d \rightarrow E_j$. Then for certain positive numbers $p_j, 1\leq j \leq n$, the following Brascamp-Lieb inequality holds for any measurable function $f_j$ on $E_j (1 \leq j \leq n)$ with some $C>0$:
\begin{equation}\label{BrascampLiebineq}
\int_{\mathbb{R}^d} \prod_{j=1}^n (f_j \circ B_j)^{p_j} \leq C \prod_{j=1}^n (\int_{E_j} f_j)^{p_j}.
\end{equation}

If this is the case, we call the minimal constant $C$ such that (\ref{BrascampLiebineq}) holds to be the \emph{Brascamp-Lieb constant} $BL(\mathbf{B}, \mathbf{p})$. Here we use $\mathbf{B}$ to denote the data $(B_1, \ldots, B_n)$ and $\mathbf{p}$ to denote the data $(p_1, \ldots, p_n)$. $(\mathbf{B}, \mathbf{p})$ is called the corresponding \emph{Brascamp-Lieb datum}. If (\ref{BrascampLiebineq}) fails for any finite $C$, we define $BL(\mathbf{B}, \mathbf{p}) = +\infty$.

Lieb \cite{lieb1990gaussian} showed
\begin{thm}
$BL(\mathbf{B}, \mathbf{p}) = BL_g (\mathbf{B}, \mathbf{p})$ where
\begin{equation}\label{defnofBLg}
BL_g(\mathbf{B}, \mathbf{p}) = \sup_{A_j : E_j \rightarrow E_j \text{ is a positive definite linear transform}} \left(\frac{\prod_{j=1}^n (\det_{E_j} A_j)^{p_j}}{\det (\sum_{j=1}^n p_j B_j^* A_j B_j)}\right)^{\frac{1}{2}}.
\end{equation}
\end{thm}

Subsequently, Bennett-Carbery-Christ-Tao \cite{bennett2008brascamp}\cite{bennett2010finite} determined a necessary and sufficient condition for $BL(\mathbf{B}, \mathbf{p}) = BL_g(\mathbf{B}, \mathbf{p}) < +\infty$. They proved that $BL(\mathbf{B}, \mathbf{p}) = BL_g(\mathbf{B}, \mathbf{p}) < +\infty$ is equivalent to the following two conditions:

(1) Scaling condition:
\begin{equation}\label{scalingcond}
\sum_j p_j \dim E_j = d;
\end{equation}

and (2) Dimension condition: For any linear subspace $V \subseteq \mathbf{R}^d$,
\begin{equation}\label{dimcond}
\dim V \leq \sum_j p_j \dim(B_j V).
\end{equation}

So we know when we can have the actually Brascamp-Lieb inequality (\ref{BrascampLiebineq}) thanks to their work.

(\ref{BrascampLiebineq}) has an equivalent version that is easier to understand intuitively. We state it in the following proposition.

\begin{prop}[Combintorial Brascamp-Lieb]\label{combinatorialBL}
Assume we have a Brascamp-Lieb datum $(\mathbf{B}, \mathbf{p})$ in $\mathbb{R}^d$. Assume $k_j = \dim \ker B_j$ and we have $n$ families of slabs. Assume the $j$-th family $\mathbb{T}_j$ consists of only $k_j$-slabs of radius $1$ whose cores are all parallel to $\ker B_j$. Also assume each $|\mathbb{T}_j|$ is finite. Then $BL(\mathbf{B}, \mathbf{p}) < +\infty$ if and only if we always have
\begin{equation}
\int_{\mathbb{R}^d} \prod_{j=1}^n (\sum_{T_j \in \mathbb{T}_j} \chi_{T_j})^{p_j} \lesssim \prod_{j=1}^n |\mathbb{T}_j|^{p_j}.
\end{equation}
\end{prop}

In light of last subsection, a perturbed version of this proposition should be true. This can indeed be proved: Recently, Bennett, Bez, Flock and Lee \cite{bennett2015stability} proved the following (non-endpoint) theorem (Theorem 1.2 of \cite{bennett2015stability}) via generalizations of Guth's method in \cite{guth2015short}.

\begin{thm}[Perturbed Brascamp-Lieb with $R^{\varepsilon}$-loss\cite{bennett2015stability}]\label{perturbedBLthmwithepsilonloss}
Assume we have a Brascamp-Lieb datum $(\mathbf{B}, \mathbf{p})$ in $\mathbb{R}^d$ with $BL(\mathbf{B}, \mathbf{p}) < +\infty$. Let $k_j = \dim \ker B_j$. Assume we have $n$ families of slabs and the $j$-th family $\mathbb{T}_j$ consists of only $k_j$-slabs of radius $1$ and size $\leq R$. Assume each $|\mathbb{T}_j|$ is finite. Also assume that each slab in the $j$-th family has its core $k_j$-plane within a $\delta$-neighborhood of $\ker B_j$ on the corresponding Grassmannian (with a given standard metric). Then when $\delta$ is sufficiently small depending on $(\mathbf{B}, \mathbf{p})$ we have
\begin{equation}
\int_{\mathbb{R}^d} \prod_{j=1}^n (\sum_{T_j \in \mathbb{T}_j} \chi_{T_j})^{p_j} \lesssim_{d, \mathbf{p}, BL(\mathbf{B}, \mathbf{p}), \varepsilon} R^{\varepsilon}\prod_{j=1}^n |\mathbb{T}_j|^{p_j}.
\end{equation}
\end{thm}

They conjectured that $R^{\varepsilon}$ can be removed here (see inequality (7) and (8) of \cite{bennett2015stability}) and we prove their conjecture in the last section of this paper.

\begin{thm}[Endpoint Perturbed Brascamp-Lieb Theorem]\label{perturbedBLthm}
Assumptions are the same as Theorem \ref{perturbedBLthmwithepsilonloss}, but no restriction on the size of slabs. When $\delta$ is sufficiently small depending on $(\mathbf{B}, \mathbf{p})$ we have
\begin{equation}
\int_{\mathbb{R}^d} \prod_{j=1}^n (\sum_{T_j \in \mathbb{T}_j} \chi_{T_j})^{p_j} \lesssim_{d, \mathbf{p}, BL(\mathbf{B}, \mathbf{p})} \prod_{j=1}^n |\mathbb{T}_j|^{p_j}.
\end{equation}
\end{thm}

Our proof will follow the same scheme as the proof of Theorem \ref{multilinearkjplanethm}. Some new difficulties present themselves and we deal with them in due course.

Like what we had in the end of last subsection, our perturbed Brascamp-Lieb theorem has some impact on the endpoint Brascamp-Lieb type restriction conjecture. This conjecture is formulated in \cite{bennett2015stability}.

\begin{conj}[Endpoint Brascamp-Lieb Type Restriction Conjecture]
Setup as we did in Conjecture \ref{endpointmultilinearkjrestriction}. But this time we don't assume that $\sum_j k_j = d$ or that $T_{\Sigma_1 (0)} \Omega_1 \wedge \cdots \wedge T_{\Sigma_n (0)} \Omega_n \neq 0$. Instead, we assume that there exists $\mathbf{p} = (p_1 , \ldots, p_n)$, $p_j > 0$, such that $BL(\mathbf{B}(\Sigma), \mathbf{p}) < \infty$ where $\mathbf{B}(\Sigma) = (T_{\Sigma_1 (0)} \Omega_1, \ldots,  T_{\Sigma_n (0)} \Omega_n)$ (here we abuse the notation a bit and for each component we really mean the linear subspace of $\mathbb{R}^d$ parallel to it). Then when the $U_j$'s are sufficiently small, we have
\begin{equation}\label{conjBLrestriction}
\int_{\mathbb{R}^d} \prod_{j=1}^n |E_j g_j|^{2p_j} \lesssim_{d, \mathbf{p}, BL(\mathbf{B}(\Sigma), \mathbf{p})} \prod_{j=1}^n \|g_j\|_{L^2 (U_j)}^{2p_j}.
\end{equation}
\end{conj}

In \cite{bennett2015stability} a local variant of (\ref{conjBLrestriction}) with $R^{\varepsilon}$-loss is proved:

\begin{equation}\label{BLrestrictionRepsilonloss}
\int_{B(0, R)} \prod_{j=1}^n |E_j g_j|^{2p_j} \lesssim_{d, \mathbf{p}, BL(\mathbf{B}(\Sigma), \mathbf{p}), \varepsilon} R^{\varepsilon} \prod_{j=1}^n \|g_j\|_{L^2 (U_j)}^{2p_j}.
\end{equation}

By Theorem \ref{perturbedBLthm} and again the same method as in the proof of Theorem 4.2 in \cite{bennett2013aspects}, we can slightly improve (\ref{BLrestrictionRepsilonloss}): There is a $\kappa = \kappa (BL(\mathbf{B}(\Sigma), \mathbf{p})) > 0$ such that

\begin{equation}\label{BLrestrictionlogloss}
\int_{B(0, R)} \prod_{j=1}^n |E_j g_j|^{2p_j} \lesssim_{d, \mathbf{p}} (\log R)^{\kappa} \prod_{j=1}^n \|g_j\|_{L^2 (U_j)}^{2p_j}.
\end{equation}

\subsection{Idea of the Proofs}

When looking for removing the factor $R^{\varepsilon}$ in Theorem \ref{multikjplanethmwithepsilonpowerloss} and Theorem \ref{perturbedBLthmwithepsilonloss}, the methods in \cite{bennett2006multilinear} or \cite{guth2015short} do not feel very appealing. Instead we will follow the path led by Guth \cite{guth2010endpoint} and try to come up with a version of the so-called polynomial method.

However, there is a major difficulty to generalize Guth's argument: Note that the zero set of one polynomial has codimension $1$. in the setting of \cite{guth2010endpoint}, because a line has dimension $1$, a line will intersect the above zero set at discrete points. And the number of such points is controlled by the degree of the polynomial. Hence we can do some counting to obtain estimates. In particular, Guth's proof relies heavily on the following \emph{cylinder estimate}.

\begin{lem}[Cylinder estimate]\label{cylinderestimatelem}
Let $T$ be a cylinder of radius $1$ and $P$ be a polynomial of degree $D$. Let $\mathbf{v}$ be a unit vector parallel to the core line of $T$. If we denote $Z(P)$ to be the zero set of $P$, then the directed volume (see Definition \ref{directvoldef1})
\begin{equation}\label{cylinderestimate}
V_{Z(P) \bigcap T} (\mathbf{v}) \lesssim_d D.
\end{equation}
\end{lem}

In the $k_j$-plane setting, the zero set of a single polynomial no longer interacts well with a $k_j$-plane: Because the latter generally has a smaller codimension, it won't intersect the former at discrete points in general. Due to this issue we cannot do counting and seem to lose our main weapon (Lemma \ref{cylinderestimatelem}).

In this paper, we deal with this difficulty and obtain our Theorem \ref{multilinearkjplanethm}. The main idea is the following: for a $k$-plane, instead of finding one single polynomial, we would like to take zero sets of $k$ polynomials to interact with it. Because the codimensions of the $k$-plane and zero sets of the $k$ polynomials add up to $d$, they will intersect at points and it is possible do counting to estimate the intersection again.

Along this line, we are taking more than $1$ polynomial to approximate an arbitrary set of $N$ cubes. We would like the zero sets of all the polynomials to be ``transverse'', with this requirement we can choose at most $d$ such polynomials. Like the original polynomial method, we would like to know how low the degrees of our polynomials can be. Guth \cite{guth2010endpoint} showed that we can always choose the first polynomial to be of degree $\lesssim_d N^{\frac{1}{d}}$. But for the second polynomial this degree bound may already be no longer valid. Think about $N$ unit squares lining up on a line in the plane $\mathbb{R}^2$. Any polynomial with degree significantly less than $N$ would have most of its zero set ``almost parallel'' to the line (see \cite{guth2014degree}) and hence two such polynomials cannot interact transversely at most of the squares. However in this example it is possible to find two transverse polynomials with degree product being $N$. More generally one can look at examples of cube grids or more generally transverse intersections of hypersurfaces and similar phenomenon happen there. Based on the above discussion, we are willing to ask the following question in the spirit of the polynomial method.

\begin{ques}\label{transverseques}
Given any $N$ disjoint unit cubes in $\mathbb{R}^d$ and $A_{\nu} > 1$ for each given cube $Q_{\nu}$, do there always exist $d$ polynomials $P_1, P_2, \ldots, P_d$ such that $\prod_{i=1}^d \deg P_i$ is roughly $\sum_{\nu} A_{\nu}$, and the zero sets of all $P_i$ have ``quantitative interaction'' $\gtrsim_d A_{\nu}$ at each of the above cubes?
\end{ques}

We notice that it looks like a ``continuous version'' of the inverse B\'{e}zout's theorem (see for example \cite{taoinversebezout}). The analogue is very difficult in algebraic geometry (see \cite{taoinversebezout} for part of the reason), and is conceivably very hard in its current continuous version too. I believe it is true and can be proved though. One can make this question rigorous by specifying the meaning of ``quantitative interaction'', see the discussion below and (\ref{finalstep2}) for a result of this flavor.

Luckily enough, we find the full power of this hard version is not needed this time. Instead, it will be equally useful to have a morally positive answer to the following ``softer'' question.

\begin{ques}\label{highdegtransverseques}
Given any $N$ disjoint unit cubes in $\mathbb{R}^d$ and $A_{\nu} > 1$ for each given cube $Q_{\nu}$, do there always exist $d$ polynomials $P_1, P_2, \ldots, P_d$ and positive numbers $\alpha_{\nu} > 1$ such that $\prod_{i=1}^d \deg P_i$ is roughly $\sum_{\nu} \alpha_{\nu}A_{\nu}$, and the zero sets of all $P_i$ have quantitative interaction $\gtrsim_d a_{\nu} A_{\nu}$ at each of the above cubes?
\end{ques}

This question is weaker than Question \ref{transverseques} because there we have the additional requirement that $a_{\nu} = 1$. In other words, we allow polynomials of higher degree here but ``with the right multiplicity''. In general, higher degree polynomials, even with the right multiplicity, do not necessarily work as well as ones with lowest possible degree, see for example some estimates in \cite{guth2014degree}. But in this application it makes no difference, as we are in a situation similar to what we have in \cite{guth2010endpoint}.

Surprisingly, it turns out that after some further refinement of the question, we find that we can take $P_1, \ldots, P_d$ all to be the same $P$ and that we can obtain $P$ by the refined polynomial method of Guth involving visibility. Once this is clear we are able to prove our theorem with a great amount of help of linear algebra and geometry.

To be more specific, we find that we can take a single nonzero polynomial (that is complicated enough to look like the product of several transverse polynomials) such that the following holds: If we denote $Z(P)$ to be the zero set of $P$, then for each relevant $Q_{\nu}$, $d$ copies of $Z(P) \bigcap Q_{\nu}$ interact in a sufficiently transverse manner. Since the $d$ copies of $Z(P) \bigcap Q_{\nu}$ interact in a very transverse way, and the copies are all the same, for any $j$ and any $Q_{\nu}$ we deduce that $k_j$ copies of $Z(P) \bigcap Q_{\nu}$ interacts sufficiently transversely with the part of the $j$-th family of slabs inside $Q_{\nu}$. But for any $j$, the $j$-th family has a limited capacity of transverse interaction with $k_j$ copies of $Z(P)$ by B\'{e}zout's Theorem. This gives us an estimate that leads to Theorem \ref{multilinearkjplanethm}.

As we saw above, we end up taking one single polynomial for $d$ times. Nevertheless, we choose to keep the entire thought process on ``$d$ transverse polynomials'' here because after all, it is how we eventually come up with the solution and the reader might find our thought process useful elsewhere. Also, Question \ref{transverseques} that remains open is still fundamental, as it's a general one concerning the polynomial approximation of any $N$ cubes. For example, it implies the existence of the polynomial in the polynomial method. Its discrete analogue is also open, see \cite{taoinversebezout}. But progress in various subcases has been made.

In the multilinear $k_j$-plane setting, our method actually proves a stronger theorem (Multilinear $k_j$-Variety Theorem \ref{multikjvarietythm}) which largely generalizes Theorem \ref{multilinearkjplanethm}. We will state its exact form after a bit more preparation. Here let me briefly describe it.

Let's take a new viewpoint. To know that a point belongs to a slab of radius $1$ is equivalent to knowing the existence of another point on the core of the slab that lies in its $1$-neighborhood. Also note that the union of all cores ($k_j$-planes) of the $j$-th family of slabs can be viewed as an algebraic variety of degree $A(j)$ and dimension $k_j$. This variety is a smoothly embedded $k_j$-manifold except some zero-volume subset. Our Theorem \ref{multilinearkjplanethm} is basically saying that the $n$ families of $k_j$-planes have limited capacity of ``transversally interaction''. We will prove that this is the general case for any $n$ algebraic varieties with total dimension $d$ in Theorem \ref{multikjvarietythm}.

This multilinear $k_j$-variety theorem immediately has interesting special cases. For instance, we have a theorem about collections of sphere shells in the flavor of Theorem \ref{multilinearkjplanethm}.

The proof of Theorem \ref{perturbedBLthm} is with almost the same machine, but we have some new difficulties: When we use this machine, we want to know how well each $k_j$-plane interacts with our polynomial. However the infomation on the Brascamp-Lieb constant seemed to be very hard to use when we try to look at things pointwisely as we do in the proof of Theorem \ref{multilinearkjplanethm}. We address this issue in Section 7 and Section 8 by proving a weaker ``integral version'' of our previous pointwise estimate. Albeit being weaker, it already leads to a proof of Theorem \ref{perturbedBLthm}.

Like the situation of Theorem \ref{multilinearkjplanethm}, Theorem \ref{perturbedBLthm} has a generalization to algebraic varieties (Theorem \ref{varietyversionofBL}) and we prove the latter to automatically imply the former. Again the current form is quite strong and interesting in its own right.

\subsection{Outline of the Paper}

In sections 2 and 3 we review Guth's polynomial method in \cite{guth2010endpoint} and develop all we needed in this subject. Section 4 consists of linear algebra preliminaries and Section 5 consists of integral geometry preliminaries. We prove Theorem \ref{multilinearkjplanethm} and Theorem \ref{affinemultilinearkjplanethm} in section 6 and Theorem \ref{perturbedBLthm} in section 8 after some preparation (section 7). We will prove them by generalizing them to versions about algebraic varieties.

\section*{Acknowledgements}

I was supported by Princeton University and the Institute for Pure and Applied Mathematics (IPAM) during the research. Part of this research was performed while I was visiting IPAM, which is supported by the National Science Foundation. I thank IPAM for their warm hospitality. I would like to thank Larry Guth for numerous discussions and quite helpful advices along the way of this project. I would also like to thank Xiaosheng Mu and Fan Zheng for very helpful and inspiring discussions. I would like to thank Jonathan Bennett and Anthony Carbery for helpful suggestions on the paper, and would in particular thank Jonathan Bennett for brining \cite{bennett2013aspects} to my attention.

\section{Polynomial with High Visibility}

In this section, we review the refined polynomial method by Guth \cite{guth2010endpoint}. We review the definition and properties of visibility and state Guth's theorem that we can find a polynomial with reasonable degree and large visibility in many cubes. Along the way we define a relevant notion, namely the \emph{fading zone}, for future convenience.

\begin{defn}\label{directvolumedefn}
In $\mathbb{R}^d$, for any compact smooth hypersurface $Z$ (possibly with boundary) and any vector $\mathbf{v}$, define the \emph{directed volume}
\begin{equation}\label{directvoldef1}
V_Z (\mathbf{v}) = \int_Z |\mathbf{v} \cdot \mathbf{n}| \mathrm{dVol}_Z
\end{equation}
where $\mathbf{n}$ is the normal vector at the corresponding point.
\end{defn}

If $\mathbf{v}$ is a unit vector, there is a formula of $V_Z (\mathbf{v})$ that is geometrically more meaningful. Let $\pi_{\mathbf{v}}$ be the orthogonal projection of $\mathbb{R}^d$ onto the subspace $\mathbf{v}^{\perp}$. Then for almost $y \in {\mathbf{v}}^{\perp}$, $|Z \bigcap \pi_{\mathbf{v}}^{-1} (y)|$ is finite and we have (see \cite{guth2010endpoint})
\begin{equation}\label{directvoldef2}
V_Z (\mathbf{v}) = \int_{\mathbf{v}^{\perp}} |Z \bigcap \pi_{\mathbf{v}}^{-1} (y)| \mathrm{d}y.
\end{equation}

\begin{defn}
The \emph{fading zone} $F(Z)$ is defined to be the set $\{\mathbf{v}: |\mathbf{v}| \leq 1, V_Z (\mathbf{v}) \leq 1\}$. It is a nonempty convex compact subset of the unit ball (see \cite{guth2010endpoint}). The visibility $Vis[Z] = \frac{1}{|F(Z)|}$.
\end{defn}

First we explain the heuristic meaning of the two concepts. Imagine that it is in midnight and we are looking at a glittering object exactly with the same shape as $Z$ from a fixed distance. To describe the situation mathematically, we can find a vector $\mathbf{v}$ such that its direction is the direction of the object and its length is the brightness of the object. Then we can intuitively think that $Z$ fades away when $\mathbf{v}$ enters the fading zone. And naturally the less visible the object is, the larger we want the fading zone to be. Hence we can define the visibility to be the inverse of the volume of the fading zone. See the beginning of Section 6 in \cite{guth2010endpoint} for how to intuitively understand visibility and a few simple examples.

It is good to keep in mind that in this paper we will mostly deal with hypersurfaces $Z$ with $V_Z (\mathbf{v}) \gtrsim_d 1$ for any unit vector $\mathbf{v}$. For hypersurfaces that don't satisfy this we will typically fix it by taking its union with several hyperplanes parallel to coordinate hyperplanes.

Apparently as long as $Z$ has finite volume, $F(Z)$ has a nonempty interior.

We are interested in polynomials and want to use the notions above to study them. Recall that the space of degree $D$ algebraic hypersurfaces in $\mathbb{R}^d$ is parametrized by $\mathbb{RP}^K$ for $K = {{D+d}\choose d} -1$ in the following way: any such hypersurface corresponds to a polynomial $P$ up to a scalar. By viewing $P$ also as the ${{D+d}\choose d}$-tuple of its coefficients we find this parametrization \cite{guth2010endpoint}. We want to think of the directed volume and the visibility as functions over $\mathbb{RP}^K$. However as Guth pointed out\cite{guth2010endpoint}, they are bad functions that may even be discontinuous.

Following Guth\cite{guth2010endpoint}, we get around this difficulty by looking at the mollified version of them. If we take the standard metric on $\mathbb{RP}^K$, we will mollify those functions over small balls around some $P \in \mathbb{RP}^K$. In the rest of this paper, we take $\varepsilon$ to be a very small positive number depending on all the constants, and in application on the set of cubes and visibility conditions. This kind of assumption is often dangerous but as we can eventually see, here it does no harm at all (mainly because all the algebraic hypersurfaces of degree $D$ satisfy the same Intersection Estimate (\ref{intersectionestimate}) uniformly), just like the case of \cite{guth2010endpoint}. There instead of the Intersection Estimate, we have the Cylinder Estimate (\ref{cylinderestimate}) as a special case counterpart.

For any $P \in \mathbb{RP}^K$, let $B(P, \varepsilon)$ be the $\varepsilon$-neighborhood of $P$. Let $Z(P)$ denote the zero set of $P$. Note that for any $P$, the set of singular points on $Z(P)$ has zero $(d-1)$-dimensional Hausdorff measure. And the rest of $Z(P)$ is a smooth embedded hypersurface by the Implicit Function Theorem.

\begin{defn}\label{defnofmollified}
For any bounded open set $U$ and any vector $\mathbf{v}$, define the \emph{mollified directed volume}
\begin{equation}\label{defineeqofaverdirectedvolume}
\overline{V}_{Z(P)\bigcap U} (\mathbf{v}) = \frac{1}{|B(P, \varepsilon)|} \int_{B(P, \varepsilon)} V_{Z(P') \bigcap U} (\mathbf{v}) \mathrm{d}P' .
\end{equation}

Define the \emph{mollified fading zone} and \emph{mollified visibility} based on the mollified directional volumes:
\begin{equation}
\overline{F}(Z(P)\bigcap U) = \{\mathbf{v}: |\mathbf{v}| \leq 1 : \overline{V}_{Z(P)\bigcap U} (\mathbf{v}) \leq 1\}
\end{equation}
\begin{equation}
\overline{Vis}[Z(P)\bigcap U]= \frac{1}{|\overline{F}(Z(P)\bigcap U)|}.
\end{equation}
\end{defn}

Like we had before for $F(Z)$, $\overline{F}(Z(P)\bigcap U)$ is a convex compact subset of the unit ball with an nonempty interior. By John's Ellipsoid Theorem \cite{john2014extremum}, for any convex set $\Gamma$ with interior, there is an ellipsoid $Ell(\Gamma)$ such that $Ell(\Gamma) \subseteq \Gamma \subseteq C_d Ell(\Gamma)$ and that $|Ell(\Gamma)| \sim_d |\Gamma|$. It is easy to see that if the convex set is symmetric about the origin (which will be the case for all convex sets considered in this paper), then we may require the ellipsoid to be symmetric about the origin too. We assume so henceforth in the paper. We call any such $Ell(\Gamma)$ an \emph{elliptical approximation} of $\Gamma$.

$\overline{V}_{Z(P)\bigcap U} (\mathbf{v})$ and $\overline{Vis}[Z(P)\bigcap U]$ are continuous with respect to $P \in \mathbb{RP}^M$ \cite{guth2010endpoint}. Guth proved the following key lemma in \cite{guth2010endpoint}.

\begin{lem}[Large Visibility on Many Cubes \cite{guth2010endpoint}]\label{Guthlargevis}
For any finite set of cubes $Q_1, \ldots, Q_N$ and non-negative integers $M(Q_i), 1\leq i \leq N$, there exists a polynomial $P$ of degree $\leq D$ (but viewed as a degree $D$ polynomial when we mollify) such that $\overline{Vis} (Z(P) \bigcap Q_k) \geq M(Q_k)$ and that $D \lesssim_d (\sum_{i=1}^N M(Q_k))^{\frac{1}{d}}$.
\end{lem}

\section{Wedge Product Estimate Based on Visibility}

As we are actually dealing with the mollification version of everything, it is convenient to have a generalized definition of visibility on any space of finite measure. Assume we have a measure space $(X, \mu)$ with $\mu(X)< \infty$ and a vector-valued measurable function $f: X \rightarrow \mathbb{R}^d$. For any vector $\mathbf{v} \in \mathbb{R}^d$ define the \emph{total absolute inner product} of $\mathbf{v}$ and $f$ (the directed volume of last section being the example we have in mind):
\begin{equation}\label{defnofgeneralvisibility}
V_{X, f} (\mathbf{v}) = \int_X |\mathbf{v} \cdot f(x)| \mathrm{d}\mu(x).
\end{equation}

Define the fading zone $F(X, f) = \{\mathbf{v}\leq 1, V_{X, f} (\mathbf{v}) \leq 1\}$ and visibility $Vis[X, f] = \frac{1}{|F(X, f)|}$. As we had in the end of last section, we have an elliptical approximation $Ell(F(X, f))$ such that $Ell(F(X, f)) \subseteq F(X, f) \subseteq C_d Ell(F(X, f))$.

Next we obtain a lower bound of a wedge product integral in terms of visibility.

\begin{thm}[Wedge Product Estimate]\label{wedgeproductestimatethmsp}
Assume that for any unit vector $\mathbf{v}$ we have $V_{X, f} (\mathbf{v}) \geq 1$. Then
\begin{equation}\label{wedgeproductestimateeqsp}
\int\cdots\int_{X^d} |\wedge_{i=1}^d f(x_i)| \mathrm{d}\mu(x_1)\mathrm{d}\mu(x_2)\cdots\mathrm{d}\mu(x_d) \gtrsim_d Vis[X, f].
\end{equation}
\end{thm}

\begin{proof}
We do induction on the dimension $d$ to prove the theorem. First observe that if $Ell(F(X, f))$ is an elliptical approximation of $F(X, f)$, then for any linear subspace $W$ of $\mathbb{R}^d$, $Ell(F(X, f)) \bigcap W$ (ellipsoid) is also an elliptical approximation of $F(X, f) \bigcap W$ by definition (this may seem problematic as the $C_d$ will vary, but for the conclusion only finite many intermediate dimensions are involved in the whole induction process and we can set the $C_d$ of them all being the same).

For $d=1$, by definition we easily see $Vis[X, f] =\frac{1}{2} \int_X |f(x)| \mathrm{d}\mu(x)$ and the conclusion holds.

Assume the conclusion holds for $d < d_0$ and $d_0 > 1$. Now we deal with the case $d = d_0$. Assume $\mathbf{v}_1, \ldots, \mathbf{v}_{d_0}$ are parallel to the semi-principal axes of any elliptical approximation $Ell(F(X, f))$, respectively, and that they form an orthonormal basis (we can arbitrarily choose a set of orthogonal semi-principal axes if there is ambiguity defining the semi-principal axes). Among them we assume $\mathbf{v}_1$ is parallel to a semi-minor axis (i.e. a shortest semi-principal axis) that has length $t_1$. Taking $\mathbf{v} = \lambda \mathbf{v}_1$ where $\lambda \sim_{d_0} t_1$ in (\ref{defnofgeneralvisibility}) we deduce
\begin{equation}\label{areaestimate}
\int_X |f(x)|\mathrm{d} \mu(x) \geq \frac{1}{t_1}.
\end{equation}

Next for any unit vector $\mathbf{v} \in \mathbb{R}^{d_0}$, we prove
\begin{equation}\label{prelimwedgeproductestimateeq}
\int\cdots\int_{X^{d_0 - 1}} |f(x_1)\wedge \cdots \wedge f(x_{d_0 -1}) \wedge \mathbf{v}| \mathrm{d}\mu(x_1)\mathrm{d}\mu(x_2)\cdots\mathrm{d}\mu(x_{d_0 -1}) \gtrsim_{d_0} t_1 \cdot Vis[X, f].
\end{equation}

Let $\pi_{\mathbf{v}^{\perp}}$ be the orthogonal projection from $\mathbb{R}^{d_0}$ to its subspace $\mathbf{v}^{\perp}$. Define $f_{\mathbf{v}^{\perp}} = \pi_{\mathbf{v}^{\perp}} \circ f$. If we identify $\mathbb{R}^{d_0 -1}$ with $\mathbf{v}^{\perp}$, $f_{\mathbf{v}^{\perp}}$ is another $(d_0 -1)$-dimensional-vector-valued function on $X$. By definition, we know for any $\mathbf{w} \in \mathbf{v}^{\perp}$, $V_{X, f}(\mathbf{w}) = V_{X, f_{\mathbf{v}^{\perp}}}( \mathbf{w})$. Hence $F(X, f_{\mathbf{v}^{\perp}}) = F(X, f) \bigcap \mathbf{v}^{\perp}$. By the previous discussion, we know we can choose $Ell(F(X, f_{\mathbf{v}^{\perp}}))$ to be $Ell(F(X, f))\bigcap \mathbf{v}^{\perp}$. But among all the $(d_0 - 1)$-dimensional sections of $Ell(F(X, f))$, the section cut by $\mathbf{v}_1^{\perp}$ has the largest volume which is $\sim_{d_0} \frac{|Ell(F(X, f))|}{t_1} = \frac{1}{t_1 \cdot Vis[X, f]}$. Hence $Vis[X, f_{\mathbf{v}^{\perp}}] = \frac{1}{|F(X, f_{\mathbf{v}^{\perp}})|} \sim_{d_0} \frac{1}{|Ell(F(X, f_{\mathbf{v}^{\perp}}))|} \gtrsim_{d_0} t_1 \cdot Vis[X, f]$. By induction hypothesis we have
\begin{eqnarray}
& \int\cdots\int_{X^{d_0 - 1}} |f(x_1)\wedge \cdots \wedge f(x_{d_0 -1}) \wedge \mathbf{v}| \mathrm{d}\mu(x_1)\mathrm{d}\mu(x_2)\cdots\mathrm{d}\mu(x_{d_0 -1})\nonumber\\
= & \int\cdots\int_{X^{d_0 - 1}} |f_{\mathbf{v}^{\perp}}(x_1)\wedge \cdots \wedge f_{\mathbf{v}^{\perp}}(x_{d_0 -1})| \mathrm{d}\mu(x_1)\mathrm{d}\mu(x_2)\cdots\mathrm{d}\mu(x_{d_0 -1})\nonumber\\
\gtrsim_{d_0} & Vis[X, f_{\mathbf{v}^{\perp}}]\nonumber\\
\gtrsim_{d_0} & t_1\cdot Vis[X, f].
\end{eqnarray}

This is (\ref{prelimwedgeproductestimateeq}).

Combining (\ref{areaestimate}) and (\ref{prelimwedgeproductestimateeq}), we have
\begin{eqnarray}
&\int\cdots\int_{X^d} |\wedge_{i=1}^d f(x_i)| \mathrm{d}\mu(x_1)\mathrm{d}\mu(x_2)\cdots\mathrm{d}\mu(x_d)\nonumber\\
= & \int_X |f(x)| (\int\cdots\int_{X^{d_0 - 1}} |f(x_1)\wedge \cdots \wedge f(x_{d_0 -1}) \wedge \frac{f(x)}{|f(x)|}| \mathrm{d}\mu(x_1)\mathrm{d}\mu(x_2)\cdots\mathrm{d}\mu(x_{d_0 -1}))\mathrm{d}\mu(x) \nonumber\\
\gtrsim_{d_0} & t_1 \cdot Vis[X, f]\cdot \int_X |f(x)| \mathrm{d}\mu(x)\nonumber\\
\gtrsim_{d_0} & Vis[X, f]
\end{eqnarray}
which concludes the induction.
\end{proof}

\section{Linear Algebra Preliminaries}

Our proof relies heavily on linear algebra. In this section we do the linear algebraic part and prove several useful lemmas.

\begin{lem}\label{equidefnofwedge}
Assume $V_1, \ldots, V_n \subseteq \mathbb{R}^d$ and $k_j = \dim V_j$ satisfies $\sum_{j=1}^n k_j = d$. Then for any orthonormal basis $\mathbf{w}_1, \ldots, \mathbf{w}_d$ we have
\begin{equation}\label{equidefnofwedgeeq}
|V_1 \wedge \cdots \wedge V_n| \lesssim_{d} \max_{\stackrel{1 \leq i_{j, h} \leq d \text{ for } 1 \leq j \leq n, k_j + 1 \leq h \leq d}{\text{ each } 1 \leq i \leq d \text{ is chosen } (n-1) \text{ times among all } i_{j, h}}} \prod_{j=1}^n |V_j \wedge \mathbf{w}_{i_{j, k_j + 1}} \wedge \cdots \wedge \mathbf{w}_{i_{j, d}}|.
\end{equation}
\end{lem}

\begin{proof}
Apparently we may assume $|V_1 \wedge \cdots \wedge V_n| > 0$. Assume $\mathbf{v}_{j, h} (1 \leq h \leq k_j)$ is an orthonormal basis of $V_j$.

We expand $\mathbf{v}_{j, h}$ under the basis $\mathbf{w}_i$ as $\mathbf{v}_{j, h} = \sum_{i=1}^d c_{j, h, i} \mathbf{w}_i$. We have
\begin{equation}
|V_1 \wedge \cdots \wedge V_n| = \left|\det \left( \begin{array}{cccc}
c_{1, 1, 1} & c_{1, 1, 2} & \cdots & c_{1, 1, d}\\
c_{1, 2, 1} & c_{1, 2, 2} & \cdots & c_{1, 2, d}\\
\cdots & \cdots & \cdots & \cdots\\
c_{1, k_1, 1} & c_{1, k_1, 2} & \cdots & c_{1, k_1, d}\\
c_{2, 1, 1} & c_{2, 1, 2} & \cdots & c_{2, 1, d}\\
\cdots & \cdots & \cdots & \cdots\\
c_{2, k_2, 1} & c_{2, k_2, 2} & \cdots & c_{2, k_2, d}\\
c_{n, 1, 1} & c_{n, 1, 2} & \cdots & c_{n, 1, d}\\
\cdots & \cdots & \cdots & \cdots\\
c_{n, k_n, 1} & c_{n, k_n, 2} & \cdots & c_{n, k_n, d}\\
\end{array} \right)\right|
\end{equation}

By Laplace's expansion theorem and the triangle inequality, we know there is a $k_j$-minor in the submatrix formed by the $(\sum_{j' < j} k_{j'} + 1)$-th row to the $(\sum_{j' \leq j} k_{j'})$-th row, such that the product of the absolute values of all $n$ determinants of such minors has absolute value $\gtrsim_{d} |V_1 \wedge \cdots \wedge V_n|$. Moreover, all chosen minors must exhaust all the $d$ columns. Denote the determinant of the $j$-th minor to be $I_j$. Then we have $|\prod_{j=1}^n I_j| \gtrsim_d |V_1 \wedge \cdots \wedge V_n|$. We choose $i_{j, k_j + 1}, \ldots, i_{j_d}$ to be all the numbers in $\{1, \ldots, d\}$ that are not any column cardinality of the $j$-th minor we selected above. Then $|I_j| = |V_j \wedge \mathbf{w}_{i_{j, k_j + 1}} \wedge \cdots \wedge \mathbf{w}_{i_{j, d}}|$ and the right hand side of (\ref{equidefnofwedgeeq}) is exactly $|\prod_{j=1}^n I_j|$.
\end{proof}

In practice, we need to deal with vectors that are not orthonormal. We have the following corollary that generalizes Lemma \ref{equidefnofwedge}.

\begin{cor}\label{equidefnofwedgestr}
Assume $V_1, \ldots, V_n \subseteq \mathbb{R}^d$ and $k_j = \dim V_j$ satisfies $\sum_{j=1}^n k_j = d$. Then for any vectors $\mathbf{w}_1, \ldots, \mathbf{w}_d \in \mathbb{R}^d$, we have
\begin{equation}\label{equidefnofwedgestreq}
\max_{\stackrel{1 \leq i_{j, h} \leq d \text{ for } 1 \leq j \leq n, k_j + 1 \leq h \leq d,}{\text{ each } 1 \leq i \leq d \text{ is chosen } (n-1) \text{ times among all } i_{j, h}}} \prod_{j=1}^n |V_j \wedge \mathbf{w}_{i_{j, k_j + 1}} \wedge \cdots \wedge \mathbf{w}_{i_{j, d}}| \gtrsim_d |V_1 \wedge \cdots \wedge V_n|\cdot|\wedge_{i=1}^d \mathbf{w}_i|^{n-1}.
\end{equation}
\end{cor}

\begin{proof}
Without loss of generality we may assume $\{\mathbf{w}_i\}$ form a basis. Consider the linear transform $T = T_{\mathbf{w}_1, \ldots, \mathbf{w}_d}: \mathbb{R}^d \rightarrow \mathbb{R}^d$ that transforms $\{\mathbf{w}_i\}$ into an orthonormal basis. Define the ``expansion factor'' $\lambda_T (V_j) = \frac{|T(\mathbf{u}_1) \wedge \cdots \wedge T(\mathbf{u}_{k_j})|}{|\mathbf{u}_1 \wedge \cdots \wedge \mathbf{u}_{k_j}|}$ for any $\mathbf{u}_1, \ldots, \mathbf{u}_{k_j}$ being a basis of $V_j$. Write $\lambda_T = \lambda_T (\mathbb{R}^d)$.

By Lemma \ref{equidefnofwedge},
\begin{equation}\label{wedgeaftertransformeq}
\max_{\stackrel{1 \leq i_{j, h} \leq d \text{ for } 1 \leq j \leq n, k_j + 1 \leq h \leq d,}{\text{ each } 1 \leq i \leq d \text{ is chosen } (d-1) \text{ times among all } i_{j, h}}} \prod_{j=1}^n |T(V_j) \wedge T(\mathbf{w}_{i_{j, k_j + 1}}) \wedge \cdots \wedge T(\mathbf{w}_{i_{j, d}})|\gtrsim_{d, n} |T(V_1) \wedge \cdots \wedge T(V_n)|.
\end{equation}

By definition it is not hard to figure out that the product on the left hand side of (\ref{wedgeaftertransformeq}) is equal to $\prod_{j=1}^n (|V_j \wedge \mathbf{w}_{i_{j, k_j + 1}} \wedge \cdots \wedge \mathbf{w}_{i_{j, d}}|\cdot \frac{\lambda_T}{\lambda_T (V_j)})$ while the right hand side is equal to $|V_1 \wedge \cdots \wedge V_n| \cdot \frac{\lambda_T}{\prod_{j=1}^n \lambda_T(V_j)}$. Note that $\lambda_T = |\mathbf{w}_1 \wedge \cdots \wedge \mathbf{w}_d|^{-1}$. The corollary then follows.
\end{proof}

We will actually use the dual form of Corollary \ref{equidefnofwedgestr}.

\begin{cor}\label{equidefnofwedgestrdual}
Assumptions are the same as Corollary \ref{equidefnofwedgestr}. We have
\begin{equation}\label{equidefnofwedgestrdualeq}
\max_{\stackrel{1 \leq i_{j, h} \leq d \text{ for } 1 \leq j \leq n, 1 \leq h \leq k_j,}{\text{ each } 1 \leq i \leq d \text{ is chosen exactly once among all } i_{j, h}}} \prod_{j=1}^n |(V_j)^{\perp} \wedge \mathbf{w}_{i_{j, 1}} \wedge \cdots \wedge \mathbf{w}_{i_{j, k_j}}| \gtrsim_d |V_1 \wedge \cdots \wedge V_n|\cdot|\wedge_{i=1}^d \mathbf{w}_i|.
\end{equation}
\end{cor}

\begin{proof}
Again without loss of generality we may assume $\{\mathbf{w}_i\}$ form a basis. Take its dual basis $\{\mathbf{u}_i\}$ such that $\mathbf{w}_{i_1} \cdot \mathbf{u}_{i_2} = \delta_{i_1, i_2}$. By Corollary \ref{equidefnofwedgestr}, we can find some $i_{j, h}, 1 \leq j \leq n, 1 \leq h \leq k_j$ such that
\begin{equation}\label{eqndualstep1}
\prod_{j=1}^n |V_j \wedge \mathbf{u}_{i_{j, k_j + 1}} \wedge \cdots \wedge \mathbf{u}_{i_{j, d}}| \gtrsim_d |V_1 \wedge \cdots \wedge V_n|\cdot|\wedge_{i=1}^d \mathbf{u}_i|^{n-1}.
\end{equation}

Choose $i_{j, k_j + 1}, \ldots, i_{j, d}$ such that $i_{j, 1} , \ldots, i_{j, d}$ form exactly the set $\{1, 2, \ldots, d\}$ for each $j$. We try to find the relation between $|(V_j)^{\perp} \wedge \mathbf{w}_{i_{j, 1}} \wedge \cdots \wedge \mathbf{w}_{i_{j, k_j}}|$ and $|V_j \wedge \mathbf{u}_{i_{j, k_j + 1}} \wedge \cdots \wedge \mathbf{u}_{i_{j, d}}|$. If we write $\mathbf{w}_i$ as column vectors, the matrix $(\mathbf{w}_i)$ and $(\mathbf{u}_i)^{T}$ are inverses to each other. Hence $|\wedge_{i=1}^d \mathbf{u}_i| = |\wedge_{i=1}^d \mathbf{w}_i|^{-1}$. For any two subspaces $X_1$ and $X_2$ of same dimension, define the angle $0 \leq \theta_{X_1, X_2} \leq \frac{\pi}{2}$ such that for any basis $\{\mathbf{x}_i\}$ of $X_1$, $\cos \theta_{X_1, X_2} = \frac{|\wedge \pi_{X_2}(\mathbf{x}_i)|}{|\wedge \mathbf{x}_i|}$ where $\pi_{X_2}$ is the orthogonal projection $\mathbb{R}^d \rightarrow X_2$. Since there is a symmetry that maps $X_1$ to $X_2$, we deduce $\theta_{X_1, X_2} = \theta_{X_2, X_1}$. Also by definition $\cos\theta_{X_1, X_2} = |X_1 \wedge (X_2)^{\perp}| = \cos\theta_{(X_1)^{\perp}, (X_2)^{\perp}}$.

Let $W_j = \text{span}\{\mathbf{w}_{i_{j, 1}}, \ldots, \mathbf{w}_{i_{j, k_j}}\}$ and $U_j = \text{span}\{\mathbf{u}_{i_{j, k_j  + 1}}, \ldots, \mathbf{u}_{i_{j, d}}\}$. We have $U_j \perp W_j$. According to the definition and properties above,
\begin{eqnarray}
&|(V_j)^{\perp} \wedge \mathbf{w}_{i_{j, 1}} \wedge \cdots \wedge \mathbf{w}_{i_{j, k_j}}|\nonumber\\
= &\cos\theta_{(V_j)^{\perp}, U_j} \cdot |U_j \wedge \mathbf{w}_{i_{j, 1}} \wedge \cdots \wedge \mathbf{w}_{i_{j, k_j}}|\nonumber\\
= &\cos\theta_{V_j, (U_j)^{\perp}} \cdot |U_j \wedge \mathbf{w}_{i_{j, 1}} \wedge \cdots \wedge \mathbf{w}_{i_{j, k_j}}|\nonumber\\
= &\cos\theta_{V_j, W_j} \cdot |U_j \wedge W_j||\mathbf{w}_{i_{j, 1}} \wedge \cdots \wedge \mathbf{w}_{i_{j, k_j}}|\nonumber\\
= &|U_j \wedge V_j||\mathbf{w}_{i_{j, 1}} \wedge \cdots \wedge \mathbf{w}_{i_{j, k_j}}|\nonumber\\
= &|V_j \wedge \mathbf{u}_{i_{j, k_j + 1}} \wedge \cdots \wedge \mathbf{u}_{i_{j, d}}|\cdot\frac{|\mathbf{w}_{i_{j, 1}} \wedge \cdots \wedge \mathbf{w}_{i_{j, k_j}}|}{|\mathbf{u}_{i_{j, k_j + 1}} \wedge \cdots \wedge \mathbf{u}_{i_{j, d}}|}\nonumber\\
\end{eqnarray}

By the definition of $\mathbf{u}_i$, we deduce that $|\mathbf{u}_{i_{j, k_j + 1}} \wedge \cdots \wedge \mathbf{u}_{i_{j, d}}|=|\text{proj}_{U_j}(\mathbf{w}_{i_{j, k_j + 1}})\wedge \cdots \wedge \text{proj}_{U_j}(\mathbf{w}_{i_{j, d}})|^{-1}$. Note that $|\mathbf{w}_{i_{j, 1}} \wedge \cdots \wedge \mathbf{w}_{i_{j, k_j}}| \cdot |\text{proj}_{U_j}(\mathbf{w}_{i_{j, k_j + 1}})\wedge \cdots \wedge \text{proj}_{U_j}(\mathbf{w}_{i_{j, d}})| = |\wedge_{i=1}^d \mathbf{w}_i|$, we have
\begin{equation}\label{eqndualstep2}
|(V_j)^{\perp} \wedge \mathbf{w}_{i_{j, 1}} \wedge \cdots \wedge \mathbf{w}_{i_{j, k_j}}| = |V_j \wedge \mathbf{u}_{i_{j, k_j + 1}} \wedge \cdots \wedge \mathbf{u}_{i_{j, d}}|\cdot |\wedge_{i=1}^d \mathbf{w}_i|.
\end{equation}

(\ref{eqndualstep1}) and (\ref{eqndualstep2}) imply the corollary.
\end{proof}

To conclude this section we compute a determinant that will be useful in the next section.

\begin{lem}\label{detlem}
Assume that $0 \leq c_j \leq d$ are integers, $1 \leq j \leq m$, satisfying $\sum_{j=1}^m c_j = d$. For any $j$, assume $\mathbf{v}_{j, 1}, \mathbf{v}_{j, 2}, \ldots, \mathbf{v}_{j, d}$ is an orthonormal basis of $\mathbb{R}^d$ (wrote as column vectors). Then we have
\begin{eqnarray}\label{deteq}
& \left|\det \left( \begin{array}{ccccccccccccc}
\mathbf{v}_{1, c_1 + 1} & \cdots & \mathbf{v}_{1, d} & \mathbf{v}_{2, c_2 + 1} & \cdots & \mathbf{v}_{2, d} &0 & \cdots & 0 & \cdots & 0 & \cdots & 0\\
\mathbf{v}_{1, c_1 + 1} & \cdots & \mathbf{v}_{1, d} & 0 & \cdots & 0 & \mathbf{v}_{3, c_3 + 1} & \cdots & \mathbf{v}_{3, d} & \cdots & 0 & \cdots & 0\\
\cdots & \cdots & \cdots & \cdots & \cdots & \cdots & \cdots & \cdots & \cdots & \cdots & \cdots & \cdots & \cdots\\
\mathbf{v}_{1, c_1 + 1} & \cdots & \mathbf{v}_{1, d} & 0 & \cdots & 0 & 0 & \cdots & 0 & \cdots & \mathbf{v}_{m, c_m + 1} & \cdots & \mathbf{v}_{m, d}
\end{array} \right)\right|\nonumber\\
= & \left|\det\left(\mathbf{v}_{1, 1} \cdots \mathbf{v}_{1, c_1} \mathbf{v}_{2, 1} \cdots \mathbf{v}_{2, 1} \cdots \mathbf{v}_{2, c_2} \cdots \mathbf{v}_{m, 1} \cdots \mathbf{v}_{m, c_m}\right)\right|.
\end{eqnarray}
\end{lem}

\begin{proof}
Let $A = \left(\mathbf{v}_{1, 1} \cdots \mathbf{v}_{1, c_1} \mathbf{v}_{2, 1} \cdots \mathbf{v}_{2, 1} \cdots \mathbf{v}_{2, c_2} \cdots \mathbf{v}_{m, 1} \cdots \mathbf{v}_{m, c_m}\right)$. Left hand side of (\ref{deteq}) is equal to
\begin{displaymath}
\left|\det \left( \begin{array}{cccccccccccccc}
I & 0 & 0 & 0 & 0 & 0 & 0 & 0 & 0 & 0 & 0 & 0 & 0 & 0\\
A &\mathbf{v}_{1, c_1 + 1} & \cdots & \mathbf{v}_{1, d} & \mathbf{v}_{2, c_2 + 1} & \cdots & \mathbf{v}_{2, d} &0 & \cdots & 0 & \cdots & 0 & \cdots & 0\\
A &\mathbf{v}_{1, c_1 + 1} & \cdots & \mathbf{v}_{1, d} & 0 & \cdots & 0 & \mathbf{v}_{3, c_3 + 1} & \cdots & \mathbf{v}_{3, d} & \cdots & 0 & \cdots & 0\\
\cdots & \cdots & \cdots & \cdots & \cdots & \cdots & \cdots & \cdots & \cdots & \cdots & \cdots & \cdots & \cdots & \cdots\\
A & \mathbf{v}_{1, c_1 + 1} & \cdots & \mathbf{v}_{1, d} & 0 & \cdots & 0 & 0 & \cdots & 0 & \cdots & \mathbf{v}_{m, c_m + 1} & \cdots & \mathbf{v}_{m, d}
\end{array} \right)\right|.
\end{displaymath}

We exchange the columns to make it look better. For simplicity let $V_j = (\mathbf{v}_{j, 1} \cdots \mathbf{v}_{j, d})$. This is an orthogonal matrix. We also define a matrix $B_j = (b_j(k, l)) (1 \leq k \leq d)$ such that: $b_j (k, l) = 1$ if $l \leq c_j$ and $k = l + \sum_{j'<j} c_{j'}$, $b_j (k, l) = 0$ otherwise. Then after rearranging the columns of the matrix above we find the determinant is equal to
\begin{displaymath}
\left|\det\left( \begin{array}{ccccc}
B_1 & B_2 & B_3 & \cdots & B_m\\
V_1 & V_2 & 0 & \cdots & 0\\
V_1 & 0 & V_3 & \cdots & 0\\
\cdots & \cdots & \cdots & \cdots & \cdots\\
V_1 & 0 & 0 & \cdots & V_m\\
\end{array} \right)\right|.
\end{displaymath}

We can multiply the $j$-th column by $V_j^{-1} = V_j^{t}$ on the right, then subtract all the $j (>1)$-th column from the first column. This keeps the determinant. Note the definition of $B_j$, if we denote $\Delta = \left(\mathbf{v}_{1, 1} \cdots \mathbf{v}_{1, c_1} -\mathbf{v}_{2, 1} \cdots -\mathbf{v}_{2, 1} \cdots -\mathbf{v}_{2, c_2} \cdots -\mathbf{v}_{m, 1} \cdots -\mathbf{v}_{m, c_m}\right)$, then the determinant is
\begin{displaymath}
\left|\det\left( \begin{array}{ccccc}
\Delta^t & B_2 V_2^t & B_3 V_3^t & \cdots & B_m V_m^t\\
0 & I & 0 & \cdots & 0\\
0 & 0 & I & \cdots & 0\\
\cdots & \cdots & \cdots & \cdots & \cdots\\
0 & 0 & 0 & \cdots & I\\
\end{array} \right)\right|.
\end{displaymath}

(\ref{deteq}) then follows directly.
\end{proof}

\section{Integral Geometry Preliminaries}

In this section we prepare some integral geometry tools for our proof of Theorem \ref{multilinearkjplanethm}. First we generalize (\ref{directvoldef2}) to the following lemma.

\begin{lem}\label{integralgeometryofintersection}
Assume in $\mathbb{R}^d$ we have $m$ smooth compact submanifolds $Z_1, Z_2, \ldots, Z_m$ (possibly with boundary) with codimensions $c_1, \ldots, c_m$ respectively. If $\sum_{j=1}^m c_j = d$ then for any measurable subset $U \subseteq \mathbb{R}^{d (m-1)} = (\mathbb{R}^d)^{m-1}$, we have
\begin{eqnarray}\label{integralgeometryofintersectioneq}
& \int_{Z_1} \int_{Z_2} \cdots \int_{Z_m} \chi_{U}(\overrightarrow{p_1 p_2}, \ldots, \overrightarrow{p_1 p_m}) |(T_{p_1} Z_1)^{\perp} \wedge \cdots \wedge (T_{p_m} Z_m)^{\perp}| \mathrm{dVol}_1 \ldots \mathrm{dVol}_m\nonumber\\
 = & \int_{\mathbf{v}_2, \ldots, \mathbf{v}_m \in \mathbb{R}^d, (\mathbf{v}_2, \ldots, \mathbf{v}_m) \in U} |(Z_1) \bigcap (Z_2 + \mathbf{v}_2) \bigcap \cdots (Z_{m-1} + \mathbf{v}_{m-1}) \bigcap (Z_m + \mathbf{v}_m)| \mathrm{d}\mathbf{v}_2 \cdots \mathrm{d}\mathbf{v}_m
\end{eqnarray}
where $p_j \in Z_j$, $T_{p_j} Z_j$ is the tangent space of $Z_j$ at $p_j$, $\mathrm{dVol}_j$ is the volume element on the $j$-th submanifold, and $Z_j + \mathbf{v}_j = \{p_j + \mathbf{v}_j : p_j \in Z_j\}$ is the translation of $Z_j$ along the vector $\mathbf{v}_j$. The $|\cdot|$ on the right hand side defines cardinality.
\end{lem}

This lemma has a lot of information so we pause a bit and go through several examples to understand it better.

When $d=2$, if $Z_1$ and $Z_2$ are two non-parallel line segments and $U$ is the whole $\mathbb{R}^2$, the integrand on the right hand side of (\ref{integralgeometryofintersectioneq}) is the characteristic function of a parallelogram generalized by $Z_1$ and $Z_2$. Hence the right hand side is the area of the parallelogram, which is easily seen to be equal to the left hand side. When $d=3$, if $Z_1$ is a line segment, $Z_2$ is a parallelogram in a plane and $U$ is the whole $\mathbb{R}^3$, the situation is totally analogous.

When $d=3$, $Z_1$ is a whole line and $Z_2$ is a smooth surface of finite area, we can take $U$ to be the point set between two planes orthogonal to $Z_1$ with distance $1$. It is a simple exercise to show that (\ref{integralgeometryofintersectioneq}) then becomes (\ref{directvoldef2}). Hence it is indeed a generalization of the latter.

Finally let's look at a more complicated example. Again take $d=3$, $U = \mathbb{R}^6$. Take a parallelepiped $\Omega = ABCD-A_1 B_1 C_1 D_1$. Take three parallelograms $Z_1 = ABCD$, $Z_2 = ABB_1 A_1$, $Z_3 = ADD_1 A_1$. Denote $\mathbf{u} = \overrightarrow{AB}, \mathbf{v} = \overrightarrow{AD}, \mathbf{w} = \overrightarrow{AA_1}$. Again the integrand on the right hand side is a characteristic function. We find it is plainly equal to $Vol(\Omega)^2$. Now the left hand side is equal to
\begin{eqnarray}
& |\mathbf{u} \times \mathbf{v}|\cdot|\mathbf{v} \times \mathbf{w}|\cdot|\mathbf{w} \times \mathbf{u}|\cdot|\frac{\mathbf{u} \times \mathbf{v}}{|\mathbf{u} \times \mathbf{v}|} \wedge \frac{\mathbf{v} \times \mathbf{w}}{|\mathbf{v} \times \mathbf{w}|} \wedge \frac{\mathbf{w} \times \mathbf{u}}{|\mathbf{w} \cdot \mathbf{u}|}|\nonumber\\
= & |(\mathbf{u} \times \mathbf{v}) \wedge (\mathbf{v} \times \mathbf{w}) \wedge (\mathbf{w} \times \mathbf{u})|\nonumber\\
= & |((\mathbf{u} \times \mathbf{v}) \times (\mathbf{v} \times \mathbf{w})) \cdot (\mathbf{w} \times \mathbf{u})|\nonumber\\
= & |(\mathbf{v} \cdot (\mathbf{u} \times \mathbf{w})) \mathbf{v} \cdot (\mathbf{w} \times \mathbf{u})|\nonumber\\
= & Vol(\Omega)^2.
\end{eqnarray}

\begin{proof}[Proof of Lemma \ref{integralgeometryofintersection}]
Without loss of generality we can assume $U$ is open and bounded. By the multilinear feature of both sides of (\ref{integralgeometryofintersectioneq}), we only need to consider this problem locally. Hence we can assume each $Z_j$ is smoothly parametrized by a domain in $\mathbb{R}^{d-c_j}$. In other words we may assume $Z_j : x_i = f_{j, i} (y_{j, 1}, \ldots, y_{j, d-c_j})$ and that the $(d-c_j)$ vectors $\mathbf{w}_{j, l}=(\frac{\partial f_{j, i}}{\partial y_{j, l}})_{1 \leq i \leq d}$ has a nonzero wedge product at any point $p_j \in Z_j$. They span the tangent space $T_{p_j} Z_j$ and will be written as column vectors below.

Look at the cartesian product $Z = Z_1 \times Z_2 \times \cdots \times Z_m \subseteq (\mathbb{R}^d)^m \cong \mathbb{R}^{dm}$. This is a smooth submanifold of dimension $\sum_{j=1}^m (d - c_j) = d(m-1)$. Use $x_{j, i} (1 \leq i \leq d)$ to denote the standard Euclidean coordinates in the $j$-th copy of $\mathbb{R}^d$ as let $v_{j, i} = x_{j, i} - x_{1, i} (j>1)$. For simplicity let $\mathbf{x}_j = (x_{j, i})_{1 \leq i \leq d}$ and $\mathbf{v}_j = (v_{j, i})_{1 \leq i \leq d}$. Notice that the right hand side of (\ref{integralgeometryofintersectioneq}) is equal to
\begin{displaymath}
\int_{Z} \chi_{U}((\mathbf{v_j})_{2 \leq j \leq m}) |\mathrm{d}\mathbf{v}_2\mathrm{d}\mathbf{v}_3\cdots\mathrm{d}\mathbf{v}_m|.
\end{displaymath}

Define the density form $\theta = |\mathrm{d}\mathbf{v}_2\mathrm{d}\mathbf{v}_3\cdots\mathrm{d}\mathbf{v}_m| = |\mathrm{d}v_{2,1}\wedge\mathrm{d}v_{2,2}\wedge\cdots\wedge\mathrm{d}v_{2,d}\wedge\cdots\wedge\mathrm{d}v_{m,1}\wedge\cdots\wedge\mathrm{d}v_{m,d}|$. On the manifold $Z$ it is a multiple of the volume density element $|\mathrm{d}V| = \prod_{j=1}^m |\wedge_{l=1}^{d-c_j} \mathbf{w}_{j, l}| |\wedge_{1 \leq j \leq m, 1\leq l \leq d-c_j}\mathrm{d}y_{j, l}|$. Next we find $\frac{\theta}{|\mathrm{d}V|}$.

We have $\frac{\theta}{|\mathrm{d}V|} = \frac{1}{\prod_{j=1}^m |\wedge_{l=1}^{d-c_j} \mathbf{w}_{j, l}|} |(\frac{\partial v_{j, l}}{\partial y_{j, l}})|$. And by change of variable we have

\begin{eqnarray}\label{quotientformcomp}
& |(\frac{\partial v_{j, i}}{\partial y_{j, l}})|\nonumber\\
= &\left|\det \left( \begin{array}{ccccccccccccc}
-\mathbf{w}_{1, 1} & \cdots & -\mathbf{w}_{1, d-c_1} & \mathbf{w}_{2, 1} & \cdots & \mathbf{w}_{2, d-c_2} & \cdots & 0 & \cdots & 0\\
\cdots & \cdots & \cdots & \cdots & \cdots & \cdots & \cdots & \cdots & \cdots & \cdots & \cdots & \cdots & \cdots\\
-\mathbf{w}_{1, 1} & \cdots & -\mathbf{w}_{1, d-c_1} & 0 & \cdots & 0 & \cdots & \mathbf{w}_{m, 1} & \cdots & \mathbf{v}_{m, d-c_m}
\end{array} \right)\right|.
\end{eqnarray}

This looks very much like the left hand side of (\ref{deteq}). Indeed, the extra negative signs does not change the determinant and can be ignored. The only essential difference here is that for each $j$, our $\{w_{j, l}\}_{1\leq l \leq d-c_j}$ is not a set of orthonormal vectors. If we do a change of variable to make them orthonormal we will extract a factor of $|\wedge_{l=1}^{d-c_j} \mathbf{w}_{j, l}|$ from right hand side of (\ref{quotientformcomp}) for each $j$. We then apply Lemma \ref{detlem} and get
\begin{equation}
\frac{\theta}{|\mathrm{d}V|} = |\wedge_{j=1}^m (T_{p_j} Z_j)^{\perp}|.
\end{equation}

Hence the right hand side of (\ref{integralgeometryofintersectioneq}) is equal to $\int_{Z} \chi_U ((\mathbf{v_j})_{2 \leq j \leq m}) |\wedge_{j=1}^m (T_{p_j} Z_j)^{\perp}| \mathrm{d}V$. Note that $\mathrm{d}V$ is the product of all $\mathrm{d} Vol_j$. This can easily be recognized as the left hand side.
\end{proof}

In application, we want to look at the case where each $V_j$ is the zero set of an algebraic variety of codimension $c_j$. Such a $V_j$ may contain singular points, but they form a subset of measure $0$ when we take the $(d-c_j)$-dimensional Hausdorff measure. Hence almost all points on $V_j$ are smooth points and we can apply our Lemma \ref{integralgeometryofintersection} to obtain the following theorem.
\begin{thm}[Intersection Estimate]\label{intersectionestimatethm}
Assume in $\mathbb{R}^d$ we have $m$ algebraic subvarieties $Z_1, Z_2, \ldots, Z_m$ with codimensions $c_1, \ldots, c_m$ and degree $s_1, \ldots, s_m$ respectively. If $\sum_{j=1}^m c_j = d$ then for any measurable subset $U \subseteq \mathbb{R}^{d (m-1)} = (\mathbb{R}^d)^{m-1}$, we have
\begin{eqnarray}\label{intersectionestimate}
& \int_{Z_1} \int_{Z_2} \cdots \int_{Z_m} \chi_U (\overrightarrow{p_1 p_2}, \ldots, \overrightarrow{p_1 p_m}) |(T_{p_1} Z_1)^{\perp} \wedge \cdots \wedge (T_{p_m} Z_m)^{\perp}| \mathrm{dVol}_1 \ldots \mathrm{dVol}_m\nonumber\\
\leq & Vol(U)\prod_{j=1}^m s_j
\end{eqnarray}
where $\mathrm{dVol}_j$ is the $(d-c_j)$-dimensional volume element on the $j$-th subvariety. Almost all $p_j \in Z_j$ are smooth points and we define $T_{p_j} Z_j$ to be the tangent space of $Z_j$ at $p_j$.
\end{thm}

\begin{proof}
(\ref{intersectionestimate}) follows directly from Lemma \ref{integralgeometryofintersection} and B\'{e}zout's Theorem.
\end{proof}

Theorem \ref{intersectionestimatethm} generalizes the cylinder estimate in \cite{guth2010endpoint} that was recorded as Lemma \ref{cylinderestimatelem} in our current paper.

\section{Proof of Theorem \ref{multilinearkjplanethm} and Theorem \ref{affinemultilinearkjplanethm}}

In this section, we prove Theorem \ref{affinemultilinearkjplanethm} and deduce Theorem \ref{multilinearkjplanethm} as a corollary. As briefly described in the introduction, we actually prove a generalized theorem about any $n$ varieties.

Basically, our multilinear $k_j$-variety theorem says that for $n$ algebraic subvarieties of $\mathbb{R}^n$, their tubular neighborhoods will provide us with an inequality similar to Theorem \ref{affinemultilinearkjplanethm} if we take their ``amount of interaction'' into account. In particular, if we take each variety to be union of $k_j$-planes we obtain Theorem \ref{affinemultilinearkjplanethm} and Theorem \ref{multilinearkjplanethm} (see the end of this section).

\begin{thm}[Multilinear $k_j$-Variety Theorem]\label{multikjvarietythm}
Assume $k_j (1 \leq j \leq n)$ satisfy $\sum_{j=1}^n k_j = d$. Assume that for $1 \leq j \leq n$, $H_j \subseteq \mathbb{R}^d$ is part of a $k_j$-dimensional algebraic subvariety of degree $A(j)$, respectively. Let $\mathrm{d} \sigma_j$ denote the $k_j$-dimensional (Hausdorff) volume measure of $H_j$. Then under this measure, almost all $y_j \in H_j$ are smooth points. For a smooth point $y_j \in H_j$, let $T_{y_j} H_j$ denote the tangent space of $H_j$ at $y_j$. Then
\begin{eqnarray}\label{multikjvarietyineq}
&\int_{\mathbb{R}^d} (\int_{H_1 \times H_2 \times \cdots \times H_n} \chi_{\{\mathrm{dist} (y_j, x) \leq 1\}} |\wedge_{j=1}^n T_{y_j} H_j| \mathrm{d}\sigma_1(y_1)\cdots \mathrm{d}\sigma_n(y_n))^{\frac{1}{n-1}} \mathrm{d}x\nonumber\\
\lesssim_{d} & \prod_{j=1}^n A(j)^{\frac{1}{n-1}}.
\end{eqnarray}
\end{thm}

We give an outline of the proof before we actually do it. For the convenience of the statement we wrote Theorem \ref{multikjvarietythm} in an integral form. However because of the truncation $\chi_{\{\mathrm{dist} (y_j, x) \leq 1\}}$ it is really of discrete flavour. In other words, around any unit cube, we only take into account the part of the varieties near this cube. By Lemma \ref{Guthlargevis}, we can find a polynomial with large visibility around each relevant cube. In the lemma, it is possible to assign different weights to different cubes in the above movement. We assign the weights according to the cubes' ``popularity'' among $H_j$, as Guth did in \cite{guth2010endpoint} for the multilinear Kakeya theorem.

We will see it does not matter if we uniformly scale all the weights. As long as the weights are large enough, we can add hyperplanes to the polynomial which do not essentially increase its degree and make its zero set satisfy the assumption of Theorem \ref{wedgeproductestimatethmsp} at each relevant cube. Then we can invoke Theorem \ref{wedgeproductestimatethmsp} for the resulting zero set $Z(P)$ at all relevant cubes to show that $d$ copies of $Z(P)$ have enough interaction there. Now around each relevant cube we are ready to assign some copies of $Z(P)$ to each variety $H_j$ and use Corollary \ref{equidefnofwedgestrdual} to show that those ``have enough interaction''. On the other hand, we can use Theorem \ref{intersectionestimatethm} to bound the amount of interaction from above. Hence we obtain a nontrivial inequality. All quantities there work out well as they supposed to be and we obtain our theorem.

\begin{proof}[Proof of Theorem \ref{multikjvarietythm}]
Apparently, we only need to prove the case where each $H_j$ is compact and take a limiting argument to complete the proof. Fix a large constant $N$ in terms of $d$, for example $N = 100e^{d}$ should be more than sufficient.

Consider the standard lattice of unit cubes in $\mathbb{R}^d$. For each cube $Q_{\nu}$ in the lattice, let $O_{\nu}$ be its center. Let
\begin{equation}
G(Q_{\nu}) = \int_{H_1 \times H_2 \times \cdots \times H_n} \chi_{\{\mathrm{dist} (y_j, O_{\nu}) \leq N\}} |\wedge_{j=1}^n T_{y_j} H_j| \mathrm{d}\sigma_1(y_1)\cdots \mathrm{d}\sigma_n(y_n).
\end{equation}

Obviously
\begin{equation}
G(Q_{\nu}) \geq \int_{H_1 \times H_2 \times \cdots \times H_n} \chi_{\{\mathrm{dist} (y_j, x) \leq 1\}} |\wedge_{j=1}^n T_{y_j} H_j| \mathrm{d}\sigma_1(y_1)\cdots \mathrm{d}\sigma_n(y_n).
\end{equation}
for any $x \in Q_{\nu}$. Hence it suffices to prove that under assumptions of Theorem \ref{multikjvarietythm}, we have
\begin{equation}\label{target}
\sum_{\nu} G(Q_{\nu})^{\frac{1}{n-1}} \lesssim_d \prod_{j=1}^m A(j)^{\frac{1}{n-1}}.
\end{equation}

We only have finite many relevant cubes $Q_{\nu}$ such that $G(Q_{\nu}) \neq 0$. Hence we can choose a huge cube of side length $S$ containing all of the relevant cubes. By Guth's Lemma \ref{Guthlargevis}, we can find a polynomial $P$ of degree $\lesssim_d S$ such that for each cube $Q_{\nu}$,
\begin{equation}\label{finalstep0}
\overline{Vis}[Z(P) \bigcap Q_{\nu}] \geq S^d G(Q_{\nu})^{\frac{1}{n-1}} (\sum_{\nu} G(Q_{\nu})^{\frac{1}{n-1}})^{-1}.
\end{equation}

Adding $\lesssim_d S$ hyperplanes to $P$ (in other words multiply $P$ by linear equations of those hyperplanes) if necessary, we may assume that for all $Q_{\nu}$ where $G(Q_{\nu}) > 0$ we have $\overline{V}_{Z(P) \bigcap Q_{\nu}} (\mathbf{v}) \geq |\mathbf{v}|$. Hence we are in a good position to use the wedge product estimate (\ref{wedgeproductestimateeqsp}).

Before we move on let me remark on a technical issue. If we do have to add hyperplanes at this point, we need to modify our Definition \ref{defnofmollified} a little bit: Assume all the hyperplanes we added form a zero set of a polynomial $P_0$. We call the original polynomial in Guth's lemma $P_{old}$ and our $P$ is actually $P_{old} P_0$. Then when we are talking about the mollified directed volume, mollified visibility, etc. around $P$, we want to look at all $P'P_0$ where $P' \in B(P_{old}, \varepsilon)$ instead of all $P' \in B(P, \varepsilon)$. For example, the definition (\ref{defineeqofaverdirectedvolume}) should now be modified to
\begin{equation}
\overline{V}_{Z(P)\bigcap U} (\mathbf{v}) = \frac{1}{|B(P_{old}, \varepsilon)|} \int_{B(P_{old}, \varepsilon)} V_{Z(P'P_0) \bigcap U} (\mathbf{v}) \mathrm{d}P' .
\end{equation}

For the rest of the section for simplicity of the notations, we deal with the case where no hyperplanes are added. For the general case the proof is identical except for proper correction of notations.


For any $y_1\in H_1\bigcap B(O_{\nu}, N), \ldots, y_n \in H_n \bigcap B(O_{\nu}, N)$, $P_1, \ldots, P_d \in B(P, \varepsilon)$(see section 2), $p_1\in Z(P_1) \bigcap B(O_{\nu}, N), \ldots, p_d \in Z(P_d)\bigcap B(O_{\nu}, N)$, by Corollary \ref{equidefnofwedgestrdual}, we can find some $i_{j, h}$ for $1 \leq j \leq n$, $1 \leq h \leq k_j$ such that
\begin{equation}
\prod_{j=1}^n |(T_{y_j} H_j)^{\perp} \wedge (T_{p_{i_{j, 1}}}Z(P_{i_{j, 1}}))^{\perp} \wedge \cdots \wedge (T_{p_{i_{j, k_j}}}Z(P_{i_{j, k_j}}))^{\perp}| \gtrsim_d |\wedge_{j=1}^n T_{y_j} H_j| \cdot |\wedge_{i=1}^d (T_{p_{i}} Z(P_i))^{\perp}|
\end{equation}
and that all $i_{j, h}$ are distinct and form exactly the set $\{1, 2, \ldots, d\}$.

Integrate over $(H_1 \bigcap B(O_{\nu}, N)) \times \cdots \times (H_n \bigcap B(O_{\nu}, N))$, we obtain
\begin{eqnarray}\label{finalstep1}
& G(Q_{\nu}) \cdot |\wedge_{i=1}^d (T_{p_i} Z(P_i))^{\perp}|\nonumber\\
\lesssim_{d}& \sum_{(i_{j, h})} \int_{H_1 \bigcap B(O_{\nu}, N)} \cdots \int_{H_n \bigcap B(O_{\nu}, N)}\nonumber\\
& \prod_{j=1}^n |(T_{y_j} H_j)^{\perp} \wedge (T_{p_{i_{j, 1}}}Z(P_{i_{j, 1}}))^{\perp} \wedge \cdots \wedge (T_{p_{i_{j, k_j}}}Z(P_{i_{j, k_j}}))^{\perp}| \mathrm{d}\sigma_1(y_1)\cdots \mathrm{d}\sigma_n(y_n).
\end{eqnarray}

Here we sum over all possible choices of $\{i_{j, h}, 1 \leq j \leq n$, $1 \leq h \leq k_j\}$ such that all $i_{j, h}$ are distinct and form exactly the set $\{1, 2, \ldots, d\}$.

Integrate (\ref{finalstep1}) over $P_1, \ldots, P_d \in B(P, \varepsilon)$ and $p_i \in Z(P_i) \bigcap B(O_{\nu}, N)$ (we abuse the notation a bit and write $\mathrm{d} p = \mathrm{d} \sigma(p)$ where $\mathrm{d}\sigma$ is the $(d-1)$-dimensional Hausdorff volume measure on $Z(P)$). Taking Definition \ref{defnofmollified} into account, we use Wedge Product Estimate Theorem \ref{wedgeproductestimatethmsp} and (\ref{finalstep0}), (\ref{finalstep1}) and deduce
\begin{eqnarray}\label{finalstep2}
& \sum_{(i_{j, h})} \frac{1}{|B(P, \varepsilon)|^d}\int\cdots\int_{B(P, \varepsilon)^d}\int_{H_1 \bigcap B(O_{\nu}, N)} \cdots \int_{H_n \bigcap B(O_{\nu}, N)} \int_{Z(P_1) \bigcap B(O_{\nu}, N)}\cdots\int_{Z(P_d) \bigcap B(O_{\nu}, N)}\nonumber\\
& \prod_{j=1}^n |(T_{y_j} H_j)^{\perp} \wedge (T_{p_{i_{j, 1}}}Z(P_{i_{j, 1}}))^{\perp} \wedge \cdots \wedge (T_{p_{i_{j, k_j}}}Z(P_{i_{j, k_j}}))^{\perp}|\nonumber\\
& \mathrm{d}p_1\cdots\mathrm{d}p_d \mathrm{d}\sigma_1(y_1)\cdots \mathrm{d}\sigma_n(y_n) \mathrm{d}P_1\cdots\mathrm{d}P_d\nonumber\\
\gtrsim_d & G(Q_{\nu}) \cdot \overline{Vis}[Z(P) \bigcap Q_{\nu}]\nonumber\\
\gtrsim_d & S^d G(Q_{\nu})^{\frac{n}{n-1}} (\sum_{\nu} G(Q_{\nu})^{\frac{1}{n-1}})^{-1}.
\end{eqnarray}

Rewrite (\ref{finalstep2}) as the following
\begin{eqnarray}\label{finalstep3}
& \sum_{(i_{j, h})}  \prod_{j=1}^n (\frac{1}{|B(P, \varepsilon)|^{k_j}}\int\cdots\int_{B(P, \varepsilon)^{k_j}}\frac{1}{S^{k_j}\cdot A(j)}\int_{H_j \bigcap B(O_{\nu}, N)} \int_{Z(P_{i_{j, 1}}) \bigcap B(O_{\nu}, N)}\cdots\int_{Z(P_{i_{j, k_j}}) \bigcap B(O_{\nu}, N)}\nonumber\\
& |(T_{y_j} H_j)^{\perp} \wedge (T_{p_{i_{j, 1}}}Z(P_{i_{j, 1}}))^{\perp} \wedge \cdots \wedge (T_{p_{i_{j, k_j}}}Z(P_{i_{j, k_j}}))^{\perp}| \mathrm{d}p_{i_{j, 1}}\cdots\mathrm{d}p_{i_{j, k_j}} \mathrm{d}\sigma_j (y_j) \mathrm{d} P_{i_{j, 1}}\cdots\mathrm{d} P_{i_{j, k_j}})\nonumber\\
\gtrsim_d & \frac{1}{\prod_{j=1}^n A(j)} G(Q_{\nu})^{\frac{n}{n-1}} (\sum_{\nu} G(Q_{\nu})^{\frac{1}{n-1}})^{-1}.
\end{eqnarray}

By the arithmetic-geometric mean inequality we deduce
\begin{eqnarray}\label{finalstep4}
& \frac{1}{(\prod_{j=1}^n A(j))^{\frac{1}{n}}} G(Q_{\nu})^{\frac{1}{n-1}} (\sum_{\nu} G(Q_{\nu})^{\frac{1}{n-1}})^{-\frac{1}{n}}\nonumber\\
\lesssim_d & \sum_{(i_{j, h})}  \sum_{j=1}^n \frac{1}{|B(P, \varepsilon)|^{k_j}}\int\cdots\int_{B(P, \varepsilon)^{k_j}}\frac{1}{S^{k_j}\cdot A(j)}\int_{H_j \bigcap B(O_{\nu}, N)} \int_{Z(P_{i_{j, 1}}) \bigcap B(O_{\nu}, N)}\cdots\int_{Z(P_{i_{j, k_j}}) \bigcap B(O_{\nu}, N)}\nonumber\\
& |(T_{y_j} H_j)^{\perp} \wedge (T_{p_{i_{j, 1}}}Z(P_{i_{j, 1}}))^{\perp} \wedge \cdots \wedge (T_{p_{i_{j, k_j}}}Z(P_{i_{j, k_j}}))^{\perp}| \mathrm{d}p_{i_{j, 1}}\cdots\mathrm{d}p_{i_{j, k_j}} \mathrm{d}\sigma_j(y_j) \mathrm{d} P_{i_{j, 1}}\cdots\mathrm{d} P_{i_{j, k_j}}\nonumber\\
\end{eqnarray}

Sum (\ref{finalstep4}) over $\nu$, and then we invoke the Intersection Estimate Theorem \ref{intersectionestimatethm} with $U = \{ (u_i)_{1\leq i \leq k_j + 1}:  u_i \in \mathbb{R}^d, \text{dist} (u_i, u_{i'}) < N^2\}$ (it suffices to choose $U$ large enough). Note that $\deg P_j = S$ and $\deg H_j = A(j)$, we have
\begin{eqnarray}\label{finalstep5}
& \frac{1}{(\prod_{j=1}^n A(j))^{\frac{1}{n}}} (\sum_{\nu} G(Q_{\nu})^{\frac{1}{n-1}}) (\sum_{\nu} G(Q_{\nu})^{\frac{1}{n-1}})^{-\frac{1}{n}}\nonumber\\
\lesssim_d & \sum_{(i_{j, h})}  \sum_{j=1}^n \frac{1}{|B(P, \varepsilon)|^{k_j}}\int\cdots\int_{B(P, \varepsilon)^{k_j}}\frac{1}{S^{k_j}\cdot A(j)}\int_{H_j} \int_{Z(P_{i_{j, 1}})}\cdots\int_{Z(P_{i_{j, k_j}})} \chi_{U}(y_j, p_{i_{j, 1}}, \ldots, p_{i_{j, k_j}})\nonumber\\
& \cdot |(T_{y_j} H_j)^{\perp} \wedge (T_{p_{i_{j, 1}}}Z(P_{i_{j, 1}}))^{\perp} \wedge \cdots \wedge (T_{p_{i_{j, k_j}}}Z(P_{i_{j, k_j}}))^{\perp}| \mathrm{d}p_{i_{j, 1}}\cdots\mathrm{d}p_{i_{j, k_j}} \mathrm{d}\sigma_j(y_j) \mathrm{d} P_{i_{j, 1}}\cdots\mathrm{d} P_{i_{j, k_j}}\nonumber\\
\lesssim_d & \sum_{(i_{j, h})}  \sum_{j=1}^n \frac{1}{|B(P, \varepsilon)|^{k_j}}\int\cdots\int_{B(P, \varepsilon)^{k_j}} 1 \mathrm{d} P_{i_{j, 1}}\cdots\mathrm{d} P_{i_{j, k_j}}\nonumber\\
\lesssim_d & 1
\end{eqnarray}
and (\ref{target}) holds.
\end{proof}

Theorem \ref{affinemultilinearkjplanethm} and Theorem \ref{multilinearkjplanethm} follow easily from Theorem \ref{multikjvarietythm}. To prove Theorem \ref{affinemultilinearkjplanethm} it suffices to prove the case where all $\rho_{j, a_j}$ are rational numbers. Then without loss of generality we may assume further that they are integers. By considering multiple copies of $U_{j, a_j}$'s, we can further assume they are all $1$. Then one just takes the $j$-th variety to be the union of the cores of the $j$-th family of slabs and apply Theorem \ref{multikjvarietythm} (after a scaling). Theorem \ref{multilinearkjplanethm} is a direct corollary of Theorem \ref{affinemultilinearkjplanethm}.

\section{An analogue of Corollary \ref{equidefnofwedgestrdual}}

In the rest of this paper, we prove Theorem \ref{perturbedBLthm}. In this section we prove a lemma (Corollary \ref{keylemofBLcasedual}) analogous to Corollary \ref{equidefnofwedgestrdual} which will be used in the proof the same way as Corollary \ref{equidefnofwedgestrdual} was used in the proof of Theorem \ref{multilinearkjplanethm}. This lemma is weaker in appearance than Corollary \ref{equidefnofwedgestrdual}, but it turns out that it serves our purpose.

\begin{defn}\label{defnofdualbody}
In $\mathbb{R}^d$, given a convex body $\Gamma$ centered at the origin, define its \emph{dual body} $\Gamma^*$ to be $\{\mathbf{v} \in \mathbb{R}^d : |(\mathbf{v}, \mathbf{x})| \leq 1 \text{ for all } x \in \Gamma\}$, where $(\cdot, \cdot)$ is the Euclidean inner product on $\mathbb{R}^d$.
\end{defn}

It is trivial by definition that if two convex bodies $\Gamma_1 \subseteq \Gamma_2$ then $\Gamma_1^* \supseteq \Gamma_2^*$.

By John's Ellipsoid Theorem, we need to mainly consider ellipsoids as examples of convex bodies. Next we develop several lemmas concerning ellipsoids. From now on, when we talk about an ellipsoid in an Euclidean space, we always assume the ellipsoid has the same dimension as the background space.

\begin{lem}\label{dualellipsoidlem}
If the $\Gamma$ in Definition \ref{defnofdualbody} is a (closed) ellipsoid centered at $O$ (the origin), then $\Gamma^*$ is also an ellipsoid centered at $O$. We call $\Gamma^*$ the \emph{dual ellipsoid} of $\Gamma$. Choose a set of principal axes of $\Gamma$ (the wording is because the choices might not be unique), we have them to be again the principal axes of $\Gamma^*$. Moreover, the lengths of the corresponding principal axes of $\Gamma$ and $\Gamma^*$, when divided by $2$, are reciprocal to each other. Hence $(\Gamma^*)^* = \Gamma$ and $Vol(\Gamma) \cdot Vol(\Gamma^*) = C_d > 0$ is a constant depending only on $d$.
\end{lem}

\begin{proof}
Trivially the dual body of the unit ball is again the unit ball. Assume $\Gamma_0$ has a dual body $\Gamma_0^*$. Then for any positive definite linear transform $A$, we have by definition
\begin{eqnarray}
(A\Gamma_0)^* = & \{\mathbf{v} \in \mathbb{R}^d : |(\mathbf{v}, A\mathbf{x})| \leq 1 \text{ for all } x \in \Gamma_0\}\nonumber\\
= & \{\mathbf{v} \in \mathbb{R}^d : |(A^*\mathbf{v}, \mathbf{x})| \leq 1 \text{ for all } x \in \Gamma_0\}\nonumber\\
= & (A^*)^{-1} \Gamma_0^*\nonumber\\
= & A^{-1} \Gamma_0^*.
\end{eqnarray}
Now we can use a positive definite linear transformation $A$ to transform the closed unit ball to our $\Gamma$, by the computation above, $\Gamma^*$ is $A^{-1}$ acting on the unit ball, so it is an ellipsoid. Also the principal axes of $\Gamma$ correspond to an orthonormal basis that diagonalize $A$. This basis also diagonalize $A^{-1}$. Hence the principal axes of $\Gamma$ are also principal axes of $\Gamma^*$. The rest of the lemma is obvious.
\end{proof}

\begin{lem}\label{intersectionellipsoidlem}
Assume a subspace $V \subseteq \mathbf{R}^d$ and $\Gamma \in \mathbb{R}^d$ is an ellipsoid centered at $O$. Let $\pi_V (\cdot)$ be the orthogonal projection onto $V$. Then $\pi_V (\Gamma^*)$ and $\Gamma \bigcap V$ are dual to each other (in $V$ with respect of the induced inner product). Note these two are both ellipsoids.
\end{lem}

\begin{proof}
If $V$ has dimension $1$, then the lemma is true by definition of the dual body (note by Lemma \ref{dualellipsoidlem}, the two ellipsoids are dual to each other).

For general $V$, by the last paragraph for any $V' \subseteq V$ of dimension $1$ we have $\pi_{V'} (\pi_V (\Gamma^*))$ and $(\Gamma \bigcap V) \bigcap V'$ are dual to each other. But given the ellipsoid $\Gamma_V = \Gamma \bigcap V \subseteq V$, apparently there is only one possible set $Y \subseteq V$ such that for any $V' \subseteq V$ of dimension $1$, $\pi_{V'} (\Gamma_V)$ and $Y \bigcap V'$ are dual to each other (since $Y \bigcap V'$ is determined by $\Gamma_V$ via this property). Now by last paragraph again, the dual of $\Gamma_V$ in $V$ is this unique $Y$. Hence $\pi_V (\Gamma^*)$ has to be this dual.
\end{proof}

\begin{lem}\label{projandintersectionlem}
For any subspace $V\subseteq \mathbf{R}^d$ of dimension $d'$, we define $\pi_V$ to be the orthogonal projection onto $V$ as usual. Then for any (closed) ellipsoid $\Gamma \subseteq \mathbb{R}^d$, we have
\begin{equation}\label{projtimesintersectionisvolume}
|\pi_{V} (\Gamma)||\Gamma \bigcap V^{\perp}| = C_d |\Gamma|
\end{equation}
where $C_{d, d'}>0$ only depends on $d$ and $d'$.
\end{lem}

\begin{proof}
It is well known that in $\mathbb{R}^d$, an ellipsoid defined by $\{\mathbf{x}: (\mathbf{x}, A\mathbf{x}) \leq 1\}$ has volume $\frac{C_d}{(\det A)^{\frac{1}{2}}}$, where $A$ is a positive definite linear transform and $(\cdot, \cdot)$ is the Euclidean inner product.
Assume $\Gamma = \{\mathbf{x}: \|T\mathbf{x}\|^2 \leq 1\}$ where $T$ is a non-degenerate linear transform. Since we can multiply $T$ by any orthogonal transform on the left, we may assume $TV^{\perp} = V^{\perp}$. Then by last paragraph,
\begin{equation}\label{ellipstep0}
|\Gamma| = \frac{C_d}{|\det T|}
\end{equation}
\begin{equation}\label{ellipstep1}
|\Gamma \bigcap V^{\perp}| = \frac{C_{d, d'}}{|\det T|_{V^{\perp}}|}.
\end{equation}

Meanwhile, $\mathbf{x} \in V$ belongs to $\pi_{V} (\Gamma)$ if and only if $\inf_{\mathbf{v} \in V^{\perp}}\|T(\mathbf{x} + \mathbf{v})\| \leq 1$. By the method of least squares, $\inf_{\mathbf{v} \in V^{\perp}}\|T(\mathbf{x} + \mathbf{v})\| = \|\pi_{(TV^{\perp})^{\perp}} (T \mathbf{x})\| = \|\pi_{V} (T \mathbf{x})\|$. Hence
\begin{equation}\label{ellipstep2}
|\pi_{V} (\Gamma)| = \frac{C_{d'}}{|\det (\pi_V \circ T|_{V})|}.
\end{equation}

Now notice $\pi V^{\perp} = V^{\perp}$. Hence when written in matrix form it is easy to verify $\det T|_{V^{\perp}} \cdot \det (\pi_V \circ T|_{V}) = \det T$. This together with (\ref{ellipstep0}), (\ref{ellipstep1}), (\ref{ellipstep2}) implies (\ref{projtimesintersectionisvolume}).
\end{proof}

Now we are ready to develop an analogue of Corollary \ref{equidefnofwedgestr}, then dualize it to obtain an analogue of Corollary \ref{equidefnofwedgestrdual}. We recall that in Section 3 we defined the total absolute inner product $V_{X, f} (\mathbf{v})$, the fading zone $F(X, f)$, visibility $Vis[X, f]$, and chose an elliptical approximation $Ell(F(X, f))$ for any measurable vector-valued function $f: X \rightarrow \mathbb{R}^d$.

\begin{lem}\label{keylemofBLcase}
Fix positive integers $d$, $\tau$ and $1 \leq k_1, \ldots, k_n < d$. Let $\mathbf{p} = (p_1, \ldots, p_n), p_j = \frac{1}{\tau}$. Assume the given integers satisfy the scaling condition $\sum_j (d-k_j) = \tau d$. Let $\mathbb{R}^d$ be the standard Euclidean space.

Assume a Brascamp-Lieb datum $(\mathbf{B}, \mathbf{p})$ satisfying that all $B_j$ are orthogonal projections from $\mathbb{R}^d$ to a subspace and that $\dim \ker B_j = k_j$.

For any measurable vector valued function $f: X \rightarrow \mathbb{R}^d$ on some measure space satisfying $V_{X, f} (\mathbf{v}) \geq 1$ for all $\mathbf{v} \in \mathbb{R}^d$, we have
\begin{equation}
\prod_{j=1}^n \left(\int_{X^{d-k_j}} |B_j \wedge f(x_1) \wedge \cdots \wedge f(x_{d-k_j})| \mathrm{d}x_1 \cdots \mathrm{d} x_{d-k_j}\right) \gtrsim_{d, \mathbf{p}} (BL(\mathbf{B}, \mathbf{p}))^{-\tau} (Vis[X, f])^{\tau}.
\end{equation}
\end{lem}

\begin{proof}
Let $E_j = B_j (\mathbb{R}^d) = (\ker B_j)^{\perp}$. Similar to the proof of Theorem \ref{wedgeproductestimatethmsp}, we define $\pi_{E_j}$ to be the orthogonal projection onto $E_j$ as before and $f_{E_j} = \pi_{E_j} \circ f$. Then
\begin{equation}\label{newwedgeprodstep1}
\int_{X^{d-k_j}} |B_j \wedge f(x_1) \wedge \cdots \wedge f(x_{d-k_j})| \mathrm{d}x_1 \cdots \mathrm{d} x_{d-k_j} = \int_{X^{d-k_j}} |f_{E_j} (x_1) \wedge \cdots \wedge f_{E_j} (x_{d-k_j})| \mathrm{d}x_1 \cdots \mathrm{d} x_{d-k_j}
\end{equation}

Similar to the proof of Theorem \ref{wedgeproductestimatethmsp}, we know $F(X, f_{E_j}) = F(X, f) \bigcap E_j$. Hence we can take $Ell(F(X, f_{E_j}))$ to be $Ell(F(X, f))\bigcap E_j$. By (\ref{newwedgeprodstep1}), Theorem \ref{wedgeproductestimatethmsp}, Lemma \ref{dualellipsoidlem} and Lemma \ref{intersectionellipsoidlem},
\begin{eqnarray}\label{newwedgeprodstep2}
& \int_{X^{d-k_j}} |B_j \wedge f(x_1) \wedge \cdots \wedge f(x_{d-k_j})| \mathrm{d}x_1 \cdots \mathrm{d} x_{d-k_j}\nonumber\\
= & \int_{X^{d-k_j}} |f_{E_j} (x_1) \wedge \cdots \wedge f_{E_j} (x_{d-k_j})| \mathrm{d}x_1 \cdots \mathrm{d} x_{d-k_j}\nonumber\\
\gtrsim_d & \frac{1}{|Ell(F(X, f))\bigcap E_j|}\nonumber\\
\sim_d & |(Ell(F(X, f))\bigcap E_j)^*|\nonumber\\
= & |\pi_{E_j} (Ell(F(X, f))^*)|.
\end{eqnarray}

Hence it suffices to prove
\begin{equation}\label{newwedgeprodstep3}
\prod_{j=1}^n |\pi_{E_j} (Ell(F(X, f))^*)| \gtrsim_{d, \mathbf{p}} (BL(\mathbf{B}, \mathbf{p}))^{-\tau} (Vis[X, f])^{\tau}.
\end{equation}

At this point we invoke the definition of $BL(\mathbf{B}, \mathbf{p})$. In (\ref{BrascampLiebineq}), we choose $f_j = \chi_{\pi_{E_j} (Ell(F(X, f))^*)}$. Then by definition $\prod_{j=1}^n (f_j \circ B_j)^{p_j} \geq \chi_{Ell(F(X, f))^*}$. Hence
\begin{eqnarray}\label{newwedgeprodstep4}
& |Ell(F(X, f))^*|\nonumber\\
\leq & \int_{\mathbb{R}^d} \prod_{j=1}^n (f_j \circ B_j)^{p_j}\nonumber\\
\leq & BL(\mathbf{B}, \mathbf{p}) \prod_{j=1}^n (\int_{E_j} f_j)^{p_j}\nonumber\\
= & BL(\mathbf{B}, \mathbf{p}) (\prod_{j=1}^n |\pi_{E_j} (Ell(F(X, f))^*)|)^{\frac{1}{\tau}}.
\end{eqnarray}

Finally by Lemma \ref{dualellipsoidlem} again, $|Ell(F(X, f))^*| \sim_d \frac{1}{|Ell(F(X, f))|} = Vis[X, f]$. Hence (\ref{newwedgeprodstep4}) implies (\ref{newwedgeprodstep3}), which concludes the proof.
\end{proof}

The dualization of Lemma \ref{keylemofBLcase} is the analogue of Corollary \ref{equidefnofwedgestrdual} that we are going to need, as advertised above.

\begin{cor}\label{keylemofBLcasedual}
Assumptions are the same as Lemma \ref{keylemofBLcase} and assume $E_j = B_j (\mathbb{R}^d) = (\ker B_j)^{\perp}$, we have
\begin{equation}\label{keylemofBLcasedualeq}
\prod_{j=1}^n \left(\int_{X^{k_j}} |E_j \wedge f(x_1) \wedge \cdots \wedge f(x_{k_j})| \mathrm{d}x_1 \cdots \mathrm{d} x_{k_j}\right) \gtrsim_{d, \mathbf{p}} (BL(\mathbf{B}, \mathbf{p}))^{-\tau} (Vis[X, f])^{n - \tau}.
\end{equation}
\end{cor}

\begin{proof}
When proving (\ref{newwedgeprodstep4}), we have actually shown that for any ellipsoid $\Gamma$,
\begin{equation}
BL(\mathbf{B}, \mathbf{p}) \cdot |\Gamma| \cdot (\prod_{j=1}^n |\pi_{E_j} (\Gamma^*)|)^{\frac{1}{\tau}} \gtrsim_d 1.
\end{equation}

By Lemma \ref{dualellipsoidlem}, Lemma \ref{intersectionellipsoidlem} and Lemma \ref{projandintersectionlem}, we have
\begin{equation}
|\pi_{E_j} (\Gamma^*)| \sim_{k_j} \frac{1}{|\Gamma \bigcap E_j|} \sim_{k_j, d} \frac{|\pi_{B_j} (\Gamma)|}{|\Gamma|}.
\end{equation}

Hence
\begin{equation}\label{laststepofduallemBL}
BL(\mathbf{B}, \mathbf{p}) \cdot |\Gamma| \cdot (\prod_{j=1}^n |\frac{|\pi_{B_j} (\Gamma)|}{|\Gamma|}|)^{\frac{1}{\tau}} \gtrsim_{d, \mathbf{p}} 1.
\end{equation}

Take $\Gamma = Ell(F(X, f))^*$. As in the proof of Lemma \ref{keylemofBLcase}, we have $|\Gamma| = Vis[X, f]$ and $\int_{X^{k_j}} |E_j \wedge f(x_1) \wedge \cdots \wedge f(x_{k_j})| \mathrm{d}x_1 \cdots \mathrm{d} x_{k_j} \gtrsim_d \pi_{B_j} (\Gamma)$. These facts and (\ref{laststepofduallemBL}) imply (\ref{keylemofBLcasedualeq}).
\end{proof}

\section{Proof of Theorem \ref{perturbedBLthm}}

We are ready to prove Theorem \ref{perturbedBLthm}. Just like the proof of Theorem \ref{multilinearkjplanethm}, we prove a stronger theorem concerning algebraic varieties. This theorem can also be considered as an analogue of Theorem \ref{multikjvarietythm}.


\begin{thm}[Variety Version of Brascamp-Lieb]\label{varietyversionofBL}
Assume we have positive integers $k_1, \ldots, k_n \leq d$ and rational numbers $p_1, \ldots, p_n > 0$. Choose a common denominator $\tau$ of all $p_j$ and assume $p_j = \frac{\tau_j}{\tau}$, $\tau_j \in \mathbf{Z}^+$ satisfying the scaling condition $\sum_j p_j (d-k_j) = d$.

Assume that for $1 \leq j \leq n$, $H_j \subseteq \mathbb{R}^d$ is part of a $k_j$-dimensional algebraic subvariety of degree $A(j)$, respectively. Let $\mathrm{d} \sigma_j$ denote the $k_j$-dimensional (Hausdorff) volume measure of $H_j$. Then under this measure, almost all $y_j \in H_j$ are smooth points. For a smooth point $y_j \in H_j$, let $T_{y_j} H_j$ denote the tangent space of $H_j$ at $y_j$.

For $(\sum_j \tau_j)$ smooth points $\mathbf{y} = (y_{1, 1}, \ldots, y_{1, \tau_1}, y_{2,1}, \ldots, y_{2, \tau_2}, \ldots, y_{n, \tau_n}), y_{j, l} \in H_j$, there exists a unique Brascamp-Lieb datum $(\mathbf{B}(\mathbf{y}), \mathbf{p}(\mathbf{y}))$ with $(\sum_j \tau_j)$ projections $B_j$ all being orthogonal projections within $\mathbb{R}^d$ as the following: Define $(\mathbf{B}(\mathbf{y}), \mathbf{p}(\mathbf{y})) = (B_{1, 1}, \ldots, B_{1, \tau_1}, B_{2,1}, \ldots, B_{2, \tau_2}, \ldots, B_{n, \tau_n}, \frac{1}{\tau}, \ldots, \frac{1}{\tau})$ such that $\ker B_{j, l} = T_{y_{j, l}} H_j$ and all components of $\mathbf{p}$ are $\frac{1}{\tau}$. Then
\begin{eqnarray}\label{varietyversionofBLineq}
&\int_{\mathbb{R}^d} (\int_{H_1^{\tau_1} \times \cdots \times H_n^{\tau_n}} \chi_{\{\mathrm{dist} (y_{j, k}, x) \leq 1\}} BL(\mathbf{B}(\mathbf{y}), \mathbf{p}(\mathbf{y}))^{-\tau}  \mathrm{d}\sigma_1(y_{1, 1})\cdots \mathrm{d}\sigma_1(y_{1, \tau_1}) \cdots \mathrm{d}\sigma_n(y_{n, \tau_n}))^{\frac{1}{\tau}} \mathrm{d}x\nonumber\\
\lesssim_{d, \tau_1, \ldots, \tau_n, \tau} & \prod_{j=1}^n A(j)^{p_j}.
\end{eqnarray}
\end{thm}

Let us explain the motivation of Theorem \ref{varietyversionofBL} before proving it. If we want to naturally generalize Theorem \ref{multikjvarietythm} to the Brascamp-Lieb setting, first of all we have to come up with a reasonable integral like the left hand side of (\ref{multikjvarietyineq}) to put on the left hand side. However the fact that in (\ref{multikjvarietyineq}) all $p_j = \frac{1}{n-1}$ no longer holds in our situation. In fact, $p_j$'s might even all be irrational numbers. A natural way would be approximating $(p_j)$ by rational tuples. this works (see below) but eventually we need all $p_j$'s to be the same to get an analogous quantity to left hand of (\ref{multikjvarietyineq}).

Another remark before we move on. It's good to keep in mind that we may assume $\tau_1 = \cdots =\tau_n = 1$ in this theorem without loss of generality. This is trivial to see. But we keep the theorem in its current form here so it would be more straightforward to apply.

\begin{proof}[Proof that Theorem \ref{varietyversionofBL} implies Theorem \ref{perturbedBLthm}]
Note that the conditions (\ref{scalingcond}) and (\ref{dimcond}) only have rational coefficients. Hence it is possible to choose $(n+1)$ different rational $\mathbf{p}'$ close enough to $\mathbf{p}$ such that the conditions (\ref{scalingcond}) and (\ref{dimcond}) are satisfied (that is, $BL(\mathbf{B}, \mathbf{p}') < + \infty$), and that $\mathbf{p}$ lies in the convex hull of those $\mathbf{p}'$. By interpolation we only need to prove the case when $\mathbf{p}$ is a rational vector.

Next in order to apply the result of Theorem \ref{varietyversionofBL} to prove Theorem \ref{perturbedBLthm}, we claim that if a Brascamp-Lieb data $(\mathbf{B}, \mathbf{p})$ such that $p_j = \frac{\tau_j}{\tau}$ where $\tau$ all $\tau_j$ are positive integers, then $BL(\mathbf{B}, \mathbf{p}) = BL(\mathbf{B}', \mathbf{p}')$, where $\mathbf{B}' = (B_1, \ldots, B_1, \ldots, B_n, \ldots, B_n)$ containing $\tau_j$ copies of $B_j$, and $\mathbf{p}' = (\frac{1}{\tau}, \ldots, \frac{1}{\tau})$. In fact, look at the definition (\ref{BrascampLiebineq}) of $BL(\mathbf{B}, \mathbf{p})$, we have
\begin{equation}
BL(\mathbf{B}', \mathbf{p}') = \sup_{\{f_{j, l}\}} \frac{\int_{\mathbb{R}^d} \prod_{j=1}^n \prod_{l=1}^{\tau_j} (f_{j, l} \circ B_j)^{\frac{1}{\tau}}}{\prod_{j=1}^n \prod_{l=1}^{\tau_j} (\int_{H_j} f_{j, l})^{\frac{1}{\tau}}}.
\end{equation}

Since we can always take $f_{j, l} = f_j$ for all $l$, we deduce $BL(\mathbf{B}', \mathbf{p}') \geq BL(\mathbf{B}, \mathbf{p})$. On the other hand, in the definition of $BL(\mathbf{B}, \mathbf{p})$ we can take $f_{j} = f_{j, l_j}$ for every possible tuple $(l_1, \ldots, l_n)$ satisfying $1 \leq l_j \leq \tau_j$ to deduce
\begin{equation}
\int_{\mathbb{R}^d} \prod_{j=1}^n (f_{j, l_j} \circ B_j)^{\frac{\tau_j}{\tau}} \leq BL(\mathbf{B}, \mathbf{p}) \prod_{j=1}^n (\int_{H_j} f_{j, l_j})^{\frac{\tau_j}{\tau}}
\end{equation}

Then we let $(l_j)$ run through all possible tuples and invoke H\"{o}lder to conclude that
\begin{equation}
\int_{\mathbb{R}^d} \prod_{j=1}^n \prod_{l=1}^{\tau_j} (f_{j, l} \circ B_j)^{\frac{1}{\tau}} \leq BL(\mathbf{B}, \mathbf{p}) \prod_{j=1}^n \prod_{l=1}^{\tau_j} (\int_{H_j} f_{j, l})^{\frac{1}{\tau}}.
\end{equation}

Hence $BL(\mathbf{B}', \mathbf{p}') \leq BL(\mathbf{B}, \mathbf{p})$. Therefore $BL(\mathbf{B}', \mathbf{p}') = BL(\mathbf{B}, \mathbf{p})$.

By Theorem 1.1 in \cite{bennett2015stability}, $BL$ is a locally bounded function. It is then not hard to derive Theorem \ref{perturbedBLthm} from Theorem \ref{varietyversionofBL} when $\mathbf{p}'$ is a fixed rational number.
\end{proof}

\begin{proof}[Proof of Theorem \ref{varietyversionofBL}]
It's plain that we may assume $\tau_1 = \cdots =\tau_n = 1$. For short we write $B_j = B_{j, 1}$ and $y_j = y_{j, 1}$.

We run the proof almost identically to the way we proved Theorem \ref{multikjvarietythm}. In the current proof, we omit some details for familiar manipulations in that proof to reduce redundancy and refer the reader to it.

Take the $N$ and set up the unit cube lattice in $\mathbb{R}^d$ as in the proof of Theorem \ref{multikjvarietythm}. Again let $O_{\nu}$ be the center of any cube $Q_{\nu}$ in the lattice. This time we define
\begin{equation}
G(Q_{\nu}) = \int_{H_1 \times \cdots \times H_n} \chi_{\mathrm{dist} (y_j, O_{\nu}) \leq N } BL(\mathbf{B}(\mathbf{y}), \mathbf{p} (\mathbf{y}))^{-\tau} \mathrm{d}\sigma_1 (y_1) \cdots \mathrm{d}\sigma_n (y_n).
\end{equation}

Similarly to the proof of Theorem \ref{multikjvarietythm}, it suffices to show
\begin{equation}\label{finalfinalgoal}
\sum_{\nu} G(Q_v)^{\frac{1}{\tau}} \lesssim_{d, n} \prod_{j=1}^n A(j)^{\frac{1}{\tau}}.
\end{equation}

Again we may assume for the moment that each $H_j$ is compact and use a limiting argument. Then we can again choose a large cube of side length $S$ that contains all the relevant cubes. Finally we can find a polynomial $P$ of degree $\lesssim_d S$ such that for each $Q_{\nu}$,
\begin{equation}
\overline{Vis}[Z(P) \bigcap Q_{\nu}] \geq S^d G(Q_{\nu})^{\frac{1}{\tau}} (\sum_{\nu} G(Q_{\nu})^{\frac{1}{\tau}})^{-1}.
\end{equation}

As before we have to make the technical comment that after adding some hyperplanes and changing the definition of $\overline{Vis}$ accordingly, we may assume for all $Q_{\nu}$ with $G(Q_{\nu}) > 0$ we have $\overline{V}_{Z(P)\bigcap Q_{\nu}} (\mathbf{v}) \geq |\mathbf{v}|$ (so that we are allowed to apply (\ref{keylemofBLcasedualeq})). We only deal with the case where no hyperplanes are added so the notation would be simpler.

Similar to what we did in the proof of Theorem \ref{multikjvarietythm}, we choose $B_j = T_{y_j} H_j$ and integrate (\ref{keylemofBLcasedualeq}) over $y_j \in H_j\bigcap B(O_{\nu}, N)$. Then we choose the measure space $X$ in (\ref{keylemofBLcasedualeq}) to be $\{p \in Z(P') \bigcap B(O_{\nu}, N): P' \in B(P, \varepsilon)\}$ (the measure is just the surface measure on each $Z(P')$ joint with the standard measure on $B(P, \varepsilon)$, which is $\mathrm{d} p \mathrm{d} P'$ where $P' \in B(P, \varepsilon)$ and $p \in Z(P') \bigcap B(O_{\nu}, N)$) and deduce
\begin{eqnarray}\label{finalfinalstep1}
& \frac{1}{|B(P, \varepsilon)|^{(n-\tau)d}} \int \cdots \int_{B(P, \varepsilon)^{(n-\tau)d}} \int_{H_1 \bigcap B(O_{\nu}, N)} \cdots \int_{H_n \bigcap B(O_{\nu}, N)} \int_{Z(P_1) \bigcap B(O_{\nu}, N)} \cdots \int_{Z(P_{(n-\tau)d}) \bigcap B(O_{\nu}, N)}\nonumber\\
& \prod_{j=1}^n |(T_{y_j} H_j)^{\perp} \wedge (T_{p_{k_1 + \cdots + k_{j-1} + 1}} Z(P_{k_1 + \cdots + k_{j-1}+ 1}))^{\perp} \wedge \cdots \wedge (T_{p_{k_1 + \cdots + k_j}} Z(P_{k_1 + \cdots + k_j}))^{\perp}| \nonumber\\
& \mathrm{d}p_1\cdots\mathrm{d}p_{(n-\tau)d} \mathrm{d}\sigma_1(y_1)\cdots \mathrm{d}\sigma_n(y_n) \mathrm{d}P_1\cdots\mathrm{d}P_{(n-\tau)d}\nonumber\\
\gtrsim_{d, n} & G(Q_{\nu}) \cdot \overline{Vis}[Z(P) \bigcap Q_{\nu}]^{n-\tau}\nonumber\\
\gtrsim_{d, n} & S^{(n-\tau)d} G(Q_{\nu})^{\frac{n}{\tau}} (\sum_{\nu} G(Q_{\nu})^{\frac{1}{\tau}})^{-(n-\tau)}.
\end{eqnarray}

Here note that since $\sum_{j=1}^n (d-k_j) = \tau d$ by assumption, we have $\sum_{j=1}^n k_j = (n-\tau) d$. We have used this fact in the above inequality chain (\ref{finalfinalstep1}).

As before we rewrite it as
\begin{eqnarray}\label{finalfinalstep2}
& \prod_{j=1}^n (\frac{1}{|B(P, \varepsilon)|^{k_j}}\int\cdots\int_{B(P, \varepsilon)^{k_j}}\frac{1}{S^{k_j}\cdot A(j)}\int_{H_j \bigcap B(O_{\nu}, N)} \int_{Z(P_1) \bigcap B(O_{\nu}, N)}\cdots\int_{Z(P_{k_j}) \bigcap B(O_{\nu}, N)}\nonumber\\
& |(T_{y_j} H_j)^{\perp} \wedge (T_{p_1}Z(P_1))^{\perp} \wedge \cdots \wedge (T_{p_{k_j}}Z(P_{k_j}))^{\perp}| \mathrm{d}p_1\cdots\mathrm{d}p_{k_j} \mathrm{d}\sigma_j (y_j) \mathrm{d} P_1\cdots\mathrm{d} P_{k_j})\nonumber\\
\gtrsim_{d, n} & \frac{1}{\prod_{j=1}^n A(j)} G(Q_{\nu})^{\frac{n}{\tau}} (\sum_{\nu} G(Q_{\nu})^{\frac{1}{\tau}})^{-(n-\tau)}.
\end{eqnarray}

By arithmetic-geometric inequality we have
\begin{eqnarray}\label{finalfinalstep3}
& \sum_{j=1}^n (\frac{1}{|B(P, \varepsilon)|^{k_j}}\int\cdots\int_{B(P, \varepsilon)^{k_j}}\frac{1}{S^{k_j}\cdot A(j)}\int_{H_j \bigcap B(O_{\nu}, N)} \int_{Z(P_1) \bigcap B(O_{\nu}, N)}\cdots\int_{Z(P_{k_j}) \bigcap B(O_{\nu}, N)}\nonumber\\
& |(T_{y_j} H_j)^{\perp} \wedge (T_{p_1}Z(P_1))^{\perp} \wedge \cdots \wedge (T_{p_{k_j}}Z(P_{k_j}))^{\perp}| \mathrm{d}p_1\cdots\mathrm{d}p_{k_j} \mathrm{d}\sigma_j (y_j) \mathrm{d} P_1\cdots\mathrm{d} P_{k_j})\nonumber\\
\gtrsim_{d, n} & \frac{1}{(\prod_{j=1}^n A(j))^{\frac{1}{n}}} G(Q_{\nu})^{\frac{1}{\tau}} (\sum_{\nu} G(Q_{\nu})^{\frac{1}{\tau}})^{-\frac{n-\tau}{n}}.
\end{eqnarray}

Like we did in the proof of Theorem \ref{multikjvarietythm}, summing over $\nu$ and applying Intersection Estimate Theorem \ref{intersectionestimatethm} with $U = \{(u_i)_{1 \leq i \leq k_j + 1}: u_i \in \mathbb{R}^d, \text{dist}(u_i, u_i ') < N^2\}$, we deduce
\begin{equation}\label{finalfinalstep4}
\frac{1}{(\prod_{j=1}^n A(j))^{\frac{1}{n}}} (\sum_{\nu} G(Q_{\nu})^{\frac{1}{\tau}}) (\sum_{\nu} G(Q_{\nu})^{\frac{1}{\tau}})^{-\frac{n-\tau}{n}} \lesssim_{d, n} 1
\end{equation}
which implies (\ref{finalfinalgoal}) and concludes the proof.

\end{proof}

\bibliographystyle{amsalpha}
\bibliography{stageref}

\providecommand{\bysame}{\leavevmode\hbox to3em{\hrulefill}\thinspace}
\providecommand{\MR}{\relax\ifhmode\unskip\space\fi MR }
\providecommand{\MRhref}[2]{%
  \href{http://www.ams.org/mathscinet-getitem?mr=#1}{#2}
}
\providecommand{\href}[2]{#2}
\begin{thebibliography}{BCCT10}

\bibitem[BBFL15]{bennett2015stability}
Jonathan Bennett, Neal Bez, Taryn~C Flock, and Sanghyuk Lee, \emph{Stability of
  the brascamp-lieb constant and applications}, arXiv preprint arXiv:1508.07502
  (2015).

\bibitem[BCCT08]{bennett2008brascamp}
Jonathan Bennett, Anthony Carbery, Michael Christ, and Terence Tao, \emph{The
  brascamp--lieb inequalities: finiteness, structure and extremals}, Geometric
  and Functional Analysis \textbf{17} (2008), no.~5, 1343--1415.

\bibitem[BCCT10]{bennett2010finite}
\bysame, \emph{Finite bounds for h\"{o}lder-brascamp-lieb multilinear
  inequalities}, Math. Res. Lett. \textbf{17} (2010), no.~4, 647--666.

\bibitem[BCT06]{bennett2006multilinear}
Jonathan Bennett, Anthony Carbery, and Terence Tao, \emph{On the multilinear
  restriction and kakeya conjectures}, Acta mathematica \textbf{196} (2006),
  no.~2, 261--302.

\bibitem[BD15]{bourgain2015proof}
Jean Bourgain and Ciprian Demeter, \emph{The proof of the $l^2$ decoupling
  conjecture}, Annals of mathematics \textbf{182} (2015), no.~1, 351--389.

\bibitem[Ben13]{bennett2013aspects}
Jonathan Bennett, \emph{Aspects of multilinear harmonic analysis related to
  transversality}, Harmonic Analysis and Partial Differential Equations
  \textbf{612} (2013), 1.

\bibitem[BG11]{bourgain2011bounds}
Jean Bourgain and Larry Guth, \emph{Bounds on oscillatory integral operators
  based on multilinear estimates}, Geometric and Functional Analysis
  \textbf{21} (2011), no.~6, 1239--1295.

\bibitem[Bou91]{bourga1991besicovitch}
Jean Bourgain, \emph{Besicovitch type maximal operators and applications to
  fourier analysis}, Geometric and Functional analysis \textbf{1} (1991),
  no.~2, 147--187.

\bibitem[Bou13a]{bourgain2013moment}
\bysame, \emph{Moment inequalities for trigonometric polynomials with spectrum
  in curved hypersurfaces}, Israel Journal of Mathematics \textbf{193} (2013),
  no.~1, 441--458.

\bibitem[Bou13b]{bourgain2013schrodinger}
\bysame, \emph{On the schr{\"o}dinger maximal function in higher dimension},
  Proceedings of the Steklov Institute of Mathematics \textbf{280} (2013),
  no.~1, 46--60.

\bibitem[CV13]{carbery2013endpoint}
Anthony Carbery and Stef{\'a}n~Ingi Valdimarsson, \emph{The endpoint
  multilinear kakeya theorem via the borsuk--ulam theorem}, Journal of
  Functional Analysis \textbf{264} (2013), no.~7, 1643--1663.

\bibitem[Gut10]{guth2010endpoint}
Larry Guth, \emph{The endpoint case of the bennett--carbery--tao multilinear
  kakeya conjecture}, Acta mathematica \textbf{205} (2010), no.~2, 263--286.

\bibitem[Gut14a]{guth2014degree}
\bysame, \emph{Degree reduction and graininess for kakeya-type sets in
  $\mathbb{R}^3$}, arXiv preprint arXiv:1402.0518 (2014).

\bibitem[Gut14b]{guth2014restriction}
\bysame, \emph{A restriction estimate using polynomial partitioning}, arXiv
  preprint arXiv:1407.1916 (2014).

\bibitem[Gut15]{guth2015short}
\bysame, \emph{A short proof of the multilinear kakeya inequality},
  Mathematical Proceedings of the Cambridge Philosophical Society, vol. 158,
  Cambridge Univ Press, 2015, pp.~147--153.

\bibitem[Joh14]{john2014extremum}
Fritz John, \emph{Extremum problems with inequalities as subsidiary
  conditions}, Traces and Emergence of Nonlinear Programming, Springer, 2014,
  pp.~197--215.

\bibitem[Lie90]{lieb1990gaussian}
Elliott~H Lieb, \emph{Gaussian kernels have only gaussian maximizers},
  Inventiones mathematicae \textbf{102} (1990), no.~1, 179--208.

\bibitem[OS82]{oberlin1982mapping}
DM~Oberlin and EM~Stein, \emph{Mapping properties of the radon-transform},
  Indiana University Mathematics Journal \textbf{31} (1982), no.~5, 641--650.

\bibitem[Tao]{taoinversebezout}
Terence Tao, \emph{A partial converse to bezout's theorem},
  https://terrytao.wordpress.com/2012/09/25/a-partial-converse-to-bezouts-theorem/.

\end{thebibliography}

Department of Mathematics, Princeton University, Princeton, NJ 08540

ruixiang@math.princeton.edu

\end{document}